\documentclass[11pt,reqno]{amsproc}
\title[]{Existence and Stability of Nonequilibrium Steady States of Nernst-Planck-Navier-Stokes Systems}

\author{Peter Constantin}
\address{Department of Mathematics, Princeton University, Princeton, NJ 08544}
\email{const@math.princeton.edu}

\author{Mihaela Ignatova}
\address{Department of Mathematics, Temple University, Philadelphia, PA 19122}
\email{ignatova@temple.edu}

\author{Fizay-Noah Lee}
\address{Program in Applied and Computational Mathematics, Princeton University, Princeton, NJ 08544}
\email{fizaynoah@princeton.edu}
\usepackage[margin=1in]{geometry}
\usepackage{amsmath, amsthm, amssymb}
\usepackage{times}
\usepackage{color}
\usepackage{hyperref}
\usepackage{comment}
\usepackage{enumerate}
\usepackage{setspace}

\newcommand{\pa}{\partial}

\newcommand{\la}{\label}
\newcommand{\fr}{\frac}
\newcommand{\na}{\nabla}
\newcommand{\be}{\begin{equation}}
\newcommand{\ee}{\end{equation}}
\newcommand{\bal}{\begin{aligned}}
\newcommand{\eal}{\end{aligned}}
\newcommand{\ba}{\begin{array}{l}}
\newcommand{\ea}{\end{array}}
\newcommand{\Rr}{{\mathbb R}}

\newtheorem{thm}{Theorem}
\newtheorem{prop}{Proposition}
\newtheorem{lemma}{Lemma}
\newtheorem{lem}{Lemma}
\newtheorem{rem}{Remark}
\newtheorem{defi}{Definition}

\renewcommand{\div}{{\mbox{div}\,}}
\newcommand{\D}{\Delta}

\newcommand{\ug}{\underline\gamma}
\newcommand{\og}{\overline\gamma}
\newcommand{\um}{\underline M}
\newcommand{\om}{\overline M}
\newcommand{\ugg}{\underline g}
\newcommand{\ogg}{\overline g}
\newcommand{\uV}{\underline V}
\newcommand{\oV}{\overline V}
\newcommand{\uuV}{\underline{\underline V}}

\newcommand{\uW}{\underline W}
\newcommand{\oc}{{\overline c}^*}
\newcommand{\ou}{{\overline u}^*}

\newcommand{\ugd}{{\underline \gamma}_\delta}
\newcommand{\ogd}{{\overline \gamma}_\delta}
\date{today}

\begin{document}

\noindent

\begin{abstract}
    We consider the Nernst-Planck-Navier-Stokes system in a bounded domain of $\Rr^d$, $d=2,3$ with general nonequilibrium Dirichlet boundary conditions for the ionic concentrations. We prove the existence of smooth steady state solutions and present a sufficient condition in terms of only the boundary data that guarantees that these solutions have nonzero fluid velocity. We show that time evolving solutions are ultimately bounded uniformly, independently of their initial size. In addition, we consider one dimensional steady states  with steady nonzero currents and show that they are  globally nonlinearly stable as solutions in a three dimensional periodic strip, if the currents are sufficiently weak. 
\end{abstract}

\maketitle

\section{Introduction}
We consider the Nernst-Planck-Navier-Stokes (NPNS) system in a connected, but not necessarily simply connected bounded domain $\Omega\subset\mathbb{R}^d$ ($d=2,3$) with smooth boundary. The system models electrodiffusion of ions in a fluid in the presence of an applied electrical potential on the boundary \cite{prob,rubibook}. In this paper, we study the case where there are two oppositely charged ionic species with valences $\pm 1$ (e.g. sodium and chloride ions). In this case, the system is given by the Nernst-Planck equations
\be
\bal
\pa_t c_1+u\cdot\na c_1=&D_1\div(\na c_1+c_1\na\Phi)\\
\pa_t c_2+u\cdot\na c_2=&D_2\div(\na c_2-c_2\na\Phi)\la{np}
\eal
\ee
coupled to the Poisson equation
\be
-\epsilon\D\Phi=c_1-c_2=\rho\la{pois}
\ee
and to the Navier-Stokes system
\be
\pa_t u+u\cdot\na u-\nu\D u+\na p=-K\rho\na\Phi,\quad \div u=0.\la{nse}
\ee
Above $c_1$ and $c_2$ are the local ionic concentrations of the cation and anion, respectively, $\rho$ is a rescaled local charge density, $u$ is the fluid velocity, and $\Phi$ is a rescaled electrical potential. The constant $K>0$ is a coupling constant given by the product of Boltzmann's constant $k_B$ and the absolute temperature $T_K$. The constants $D_i$ are the ionic diffusivities, $\epsilon>0$ is a rescaled dielectric permittivity of the solvent proportional to the square of the Debye length, and $\nu>0$ is the kinematic viscosity of the fluid. The dimensional counterparts of $\Phi$ and $\rho$ are given by $(k_BT_k/e)\Phi$ and $e\rho$, respectively, where $e$ is elementary charge. 

It is well known that for certain \textit{equilibrium boundary conditions,} (see \eqref{eqc1} and  \eqref {eqc2} below) the NPNS system (\ref{np})-(\ref{nse}) admits a unique steady solution, with vanishing velocity $u^* =0$, and with concentrations $c_i^*$ related to $\Phi^*$ which uniquely solves a nonlinear Poisson-Boltzmann equation 
\be
\bal
-\epsilon\D\Phi^* = c_1^*-c_2^*\\
c_1^* = Z_1^{-1} e^{-\Phi^*}\\
c_2^* = Z_2^{-1} e^{\Phi^*}
\eal
\la{PB}
\ee
with $Z_i>0$ constant for $i=1,2$. 
The equilibria of the Nernst-Planck-Navier-Stokes system are unique minimizers of a total energy that is nonincreasing in time on time dependent solutions \cite{ci}. For equilibrium boundary conditions  it is known that for $d=2$ the unique steady states are globally stable \cite{bothe,ci} and for $d=3$ locally stable \cite{np3d,ryham}. The equilibrium boundary conditions include the cases where $c_i$ obey blocking boundary conditions
\be
(\pa_nc_1+c_1\pa_n\Phi)_{|\pa\Omega}=(\pa_nc_2 - c_2\pa_n\Phi)_{|\pa\Omega}=0\la{eqc1}
\ee
(here, $\pa_n$ is the normal derivative) and $\Phi$ obeys Dirichlet, Neumann, or Robin boundary conditions. Also included is the case where $c_i$ obey a mix of blocking and Dirichlet boundary conditions and $\Phi$ obeys Dirichlet boundary conditions in such a way that the electrochemical potentials
\be
\mu_1=\log c_1+\Phi,\quad\mu_2=\log c_2-\Phi
\ee
are each constant on the boundary portions where $c_i$ obey Dirichlet boundary conditions. That is, if $c_i$ satisfy Dirichlet boundary conditions on $S_i\subset\pa\Omega$ (possibly $S_i=\pa\Omega)$, then boundary conditions such that
\be
{\mu_1}_{|S_1}=\text{constant},\quad{\mu_2}_{|S_2}=\text{constant}\la{eqc2}
\ee
yield an equilibrium boundary condition. In general, arbitrary deviations from such situations can produce instabilities and even chaotic behavior for time dependent solutions to NPNS \cite{davidson,kang,pham,rubinstein,rubisegel,rubizaltz,zaltzrubi}. Furthermore, steady states in nonequilibrium configurations are not known in general, and not known to always be unique \cite{mock,park}.

The various boundary conditions for $c_i$ and $\Phi$ all have physical interpretations, and we refer the reader to \cite{bothe,ci,cil,davidson,mock,schmuck} for relevant discussions. In this paper, we consider only Dirichlet boundary conditions for both $c_i$ and $\Phi$, together with no-slip boundary conditions for $u$,
\be
\bal
{c_i}_{|\pa\Omega}&=\gamma_i>0\\
\Phi_{|\pa\Omega}&=W\\
u_{|\pa\Omega}&=0.\la{BC}
\eal
\ee

For simplicity, we assume $\gamma_i,W\in C^\infty(\pa\Omega)$, but we do not restrict their size, nor do we require them to be constant. For $c_i$, the Dirichlet boundary conditions model, for example, ion-selectivity at an ion-selective membrane or some fixed concentration of ions at the boundary layer-bulk interface. Dirichlet boundary conditions for $\Phi$ model an applied electric potential on the boundary.

There is a large literature on the well-posedness of the time dependent NPNS system \cite{bothe,ci,np3d,cil,fischer,fnl,liu,ryham,schmuck}, as well as the uncoupled Nernst-Planck \cite{biler,biler2,choi,gaj,gajewski,mock} and Navier-Stokes systems \cite{cf,temam}. Some of the aforementioned studies, in addition to \cite{EN,gajstab, park}, study several aspects of the steady state Nernst-Planck equations including existence, uniqueness, stability, and asymptotic behavior.

Thus far, in the context of NPNS, steady states have mostly been studied in the case of equilibrium boundary conditions. In these cases, the corresponding steady states are the unique Nernst-Planck steady states, together with zero fluid flow $u^*\equiv 0$. In this paper in Section \ref{ssnpns}, Theorem \ref{thm1}, we prove the existence of smooth steady state solutions to the NPNS system (\ref{np})-(\ref{nse}) subject to arbitrary (large data)  Dirichlet boundary conditions (\ref{BC}). In addition, we derive a sufficient condition, depending solely on the boundary data such that the steady state solution has nonzero fluid flow $u^*\not\equiv 0$ (Theorem \ref{unot0}). Thus the two main results of Section ~\ref{ssnpns} give the existence of steady states for NPNS that are not obtained by existing theory for Nernst-Planck equations, and include in particular the steady solutions with nonzero flow for which instability and chaos have been observed experimentally and numerically.

In Section \ref{lt} we consider the time dependent solutions of the Nernst-Planck-Stokes system in 3D and show, using a maximum principle that solutions obey long time bounds that are independent of the size of the initial data. This result is also valid for the NPNS system in 2D and the NPNS system in 3D under the assumption of globally bounded smooth velocities.  The maximum principle is for a two-by-two parabolic system with unequal diffusivities. The bound applies for situations far away from equilibrium, when the solutions have nontrivial dynamics, and establishes the existence of an absorbing ball. This is a first step in proving the existence and finite dimensionality of the global attractor, a task that will be pursued elsewhere.

In Section \ref{GS}, we consider the Nernst-Planck-Stokes (NPS) system in the periodic channel $\Omega=(0,L)\times\mathbb{T}^2$  with piecewise constant (i.e. constant at $x=0$ and $x=L$, respectively) boundary conditions. In this case, we derive sufficient conditions, depending only on boundary data and parameters, such that NPS admits a one dimensional, globally stable steady state solution that corresponds to a \textit{steady (nonzero) current} solution with the fluid at rest. These are non-Boltzmann states (whose currents identically vanish). The stability condition can be thought of as a smallness condition on the magnitude of the ionic currents or, equivalently, as a small perturbation from equilibrium condition. The main result of Section \ref{GS}, Theorem \ref{globalstab}, is preceded by an analysis of one dimensional steady currents. 
\subsection{Notation}
Unless otherwise stated, we denote by $C$ a positive constant that depends only on the parameters of the system, the domain, and the initial and boundary conditions. The value of $C$ may differ from line to line.

We denote by $L^p=L^p(\Omega)$ the Lebesgue spaces and by $W^{s,p}=W^{s,p}(\Omega)$, $H^s=H^s(\Omega)=W^{s,2}$ the Sobolev spaces. We denote by $\left<\cdot,\cdot\right>$ the $L^2(\Omega)$ inner product, and we write $dV=dx\,dy\,dz$ for the volume element in three dimensions, and $dx$ in one dimension. We write $dS$ for the surface element.

We denote by $\pa_x, \pa_y, \pa_z, \pa_t$ the partial derivaties with respect to $x,y,z,t$, respectively, and also use $\pa_x$ to mean $\fr{d}{dx}$ in a one dimensional setting.

We denote $z_1=1, z_2=-1$ for the valences of the ionic species.
\section{Steady State Nernst-Planck-Navier-Stokes}\la{ssnpns}
We consider the steady state Nernst-Planck-Navier-Stokes sytem
\begin{align}
u\cdot\na c_1&=\div(\na c_1+c_1\na\Phi)\la{s}\\
u\cdot\na c_2&=\div(\na c_2-c_2\na\Phi)\\
-\D\Phi&=c_1-c_2=\rho\\
u\cdot\na u-\D u+\na p &= -\rho\na\Phi\la{sstokes}\\
\div u&=0\la{S}
\end{align}
on a smooth, connected, bounded domain $\Omega\subset\mathbb{R}^d$ ($d=2,3)$ together with boundary conditions
\begin{align}
    {c_i}_{|\pa\Omega}&=\gamma_i>0\la{b}\\
    \Phi_{|\pa\Omega}&=W\\
    u_{|\pa\Omega}&=0\la{B}
\end{align}
with $\gamma_i,W\in C^\infty(\pa\Omega)$ not necessarily constant. In the above system, we have taken $D_i=\epsilon=\nu=K=1$, as the values of these parameters do not play a significant role in the results of this section. In this section we first prove the existence of a smooth solution to (\ref{s})-(\ref{S}) with boundary conditions (\ref{b})-(\ref{B}). Then, we derive a sufficient condition depending on just $\gamma_i$ and $W$ and their derivatives that guarantees that any steady state solution $(c_i^*,\Phi^*,u^*)$ of (\ref{s})-\ref{S}) with (\ref{b})-(\ref{B}) has nonzero fluid flow i.e. $u^*\not\equiv 0$. 

\begin{thm}
For arbitrary boundary conditions (\ref{b})-(\ref{B}) on a smooth, connected, bounded domain 
$\Omega\subset\mathbb{R^d}$ ($d=2,3$), there exists a smooth solution $(c_i,\Phi,u)$ of the steady state Nernst-Planck-Navier-Stokes system (\ref{s})-(\ref{S})  such that $c_i\ge 0$.\la{thm1}
\end{thm}

\begin{rem}
For the steady state Nernst-Planck system, any regular enough solution is necessarily nonnegative (i.e. $c_i\ge 0$) if we assume $\gamma_i\ge 0$. This follows from the fact that the quantities $c_1e^\Phi$, $c_2e^{-\Phi}$ each satisfy a maximum principle (see Section \ref{odss}). However, in the case of steady state NPNS, where such a maximum principle does not hold, the nonnegativity of $c_i$ must be built into the construction.
\end{rem}
\begin{proof}
Throughout the proof, we assume $d=3$. Some steps are streamlined if we assume $d=2$, but the proof for $d=3$ nonetheless works for $d=2$.

The proof consists of two main steps. We first show the existence of a solution to a parameterized approximate NPNS system. Then, we extract a convergent subsequence and show that the limit satisfies the original system. The approximate system is given by :
\begin{align}
0&=\div(\na c_1^\delta+\chi_\delta(c_1^\delta)\na\Phi^\delta-u^\delta\chi_\delta(c_1^\delta))\la{ap}\\
0&=\div(\na c_2^\delta-\chi_\delta(c_2^\delta)\na\Phi^\delta-u^\delta\chi_\delta(c_2^\delta))\la{ap'}\\
-\D\Phi^\delta&=\chi_\delta(c_1^\delta)-\chi_\delta(c_2^\delta)={\rho}^\delta\la{ap''}\\
u^\delta\cdot\na u^\delta-\D u^\delta+\na p^\delta &= -{\rho}^\delta\na\Phi^\delta\quad \la{ap'''}\\
\div u^\delta&=0\la{Sd}
\end{align}
with boundary conditions
\begin{align}
    {c_i^\delta}_{|\pa\Omega}&=\gamma_i>0\\
    {\Phi^\delta}_{|\pa\Omega}&=W\\
    {u^\delta}_{|\pa\Omega}&=0\la{AP}.
\end{align}
Here, $\chi_\delta$ is a smooth cutoff function, which converges pointwise to the following function as $\delta\to 0$,
\be
l:y\mapsto\begin{cases}y,&\quad y\ge 0\\
0,&\quad y\le 0.\end{cases}\la{l}
\ee
We define $\chi_\delta$ by first fixing a smooth, nondecreasing function $\chi:\mathbb{R}\to \mathbb{R}^+$ such that
\be
\chi:y\mapsto\begin{cases}y,&\quad y\ge 1\\
0,&\quad y\le 0.\end{cases}\la{chi}
\ee 
Then, we set $\chi_\delta(y)=\delta\chi(\fr{y}{\delta})$. We state below some elementary properties of $\chi_\delta$:\\
\indent 1) $\chi_\delta\ge 0$\\
\indent 2) $\chi_\delta(y)=y$ for $y\ge\delta$, $\chi^\delta(y)=0$ for $y\le 0$\\
\indent 3) $\chi_\delta$ is nondecreasing\\
\indent 4) $|\chi'_\delta(y)|\le a$ where $a=\sup\chi'$ and so $\chi_\delta(y)\le a|y|, |\chi_\delta(x)-\chi_\delta(y)|\le a|x-y|$.

\noindent The existence of a solution to (\ref{ap})-(\ref{AP}) follows as an application of Schaefer's fixed point theorem \cite{evans}, which we state below:
\begin{thm}
Suppose $X$ is a Banach space and $E:X\to X$ is continuous and compact. If the set
\be
\{v\in X\,|\, v=\lambda E(v)\,\text{for some } 0\le\lambda\le 1\}\la{hl}
\ee
is bounded in $X$, then $E$ has a fixed point.\la{fp}
\end{thm}

In order to apply this fixed point theorem, we first reformulate the problem (\ref{ap})-(\ref{AP}). Letting $\Gamma_i$ be the unique harmonic function on $\Omega$ satisfying ${\Gamma_i}_{|\pa\Omega}=\gamma_i$, and introducing
\begin{align}
q_i^\delta&=c_i^\delta-\Gamma_i\la{qc}
\end{align}
we rewrite (\ref{ap})-(\ref{ap''}),
\begin{align}
    -\D q_1^\delta&=\div(-u^\delta \chi_\delta(c_1^\delta)+\chi_\delta(c_1^\delta)\na(-\D_W)^{-1}{\rho}^\delta)=R_1^\delta(q_1^\delta,q_2^\delta,u^\delta)\la{q1e}\\
    -\D q_2^\delta&=\div(-u^\delta \chi_\delta(c_2^\delta)-\chi_\delta(c_2^\delta)\na(-\D_W)^{-1}{\rho}^\delta)=R_2^\delta(q_1^\delta,q_2^\delta,u^\delta)\la{q2e}
\end{align}
where $(-\D_W)^{-1}$ maps $g$ to the unique solution $f$ of 
\be
-\D f=g\text{ in }\Omega, \quad f_{|\pa\Omega}=W.
\ee
Above we view $R_i^\delta$ as functions of $q_i^\delta$ and $u^\delta$, with $c_i^\delta$ and ${\rho}^\delta$ related to $q_i^\delta$ via (\ref{qc}) and ${\rho}^\delta=\chi_\delta(c_1^\delta)-\chi_\delta(c_2^\delta)$. Thus we write (\ref{q1e})-(\ref{q2e}) as
\begin{align}
    q_i^\delta=(-\D_D)^{-1}R_i^\delta(q_1^\delta,q_2^\delta,u^\delta), \quad i=1,2\la{Ri}
\end{align}
where $-\D_D$ is the Laplace operator on $\Omega$ associated with homogeneous Dirichlet boundary conditions. As for the Navier-Stokes subsystem, we first project the equations onto the space of divergence free vector fields using the Leray projection \cite{cf}
\be
\mathbb{P}:(L^2(\Omega))^3\to H
\ee
where $H$ is the closure of 
\be
\mathcal{V}=\{f\in (C_0^\infty(\Omega))^3\,|\,\div f=0\}
\ee
in $(L^2(\Omega))^3$ and is a Hilbert space endowed with the $L^2$ inner product. Then (\ref{ap'''}), (\ref{Sd}) is given by
\be
Au^\delta+B(u^\delta, u^\delta)=-\mathbb{P}(\rho^\delta\na\Phi^\delta)\la{AB}
\ee
where
\begin{align}
A&=\mathbb{P}(-\D):\mathcal{D}(A)=(H^2(\Omega))^3\cap V\to H\\
V&=\text{closure of }\mathcal{V}\text{ in } (H^1_0(\Omega))^3=\{f\in (H_0^1(\Omega))^3\,|\,\div f=0\}.\la{V}
\end{align}
As it is well known, the Stokes operator $A$ is invertible and $A^{-1}:H\to\mathcal{D}(A)$ is bounded and self-adjoint on $H$ and compact as a mapping from $H$ into $V$. The space $V$ is a Hilbert space endowed with the Dirichlet inner product
\be
\langle f,g\rangle_V=\int\na f:\na g\,dV.
\ee

For $f,g\in V$, $B(f,g)=\mathbb{P}(f\cdot\na g)$ and $B$ may be viewed as a continuous, bilinear mapping such that
\be
B:(f,g)\in V\times V\mapsto \left(h\in V\mapsto \int_\Omega (f\cdot\na g)\cdot h\,dV\right)\in V'
\ee
where $V'$ is the dual space of $V$. We note that we may also view $A$ as an invertible mapping $A:V\to V'$. It is with this viewpoint that we write (\ref{AB}) as
\begin{align}
u^\delta&=A^{-1}R_u^\delta(q_1^\delta,q_2^\delta,u^\delta)\\
R_u^\delta(q_1^\delta,q_2^\delta,u^\delta)&=-(B(u^\delta,u^\delta)+\mathbb{P}(\rho^\delta\na(-\D_W)^{-1}{\rho}^\delta)).\la{Ru}
\end{align}
Thus, setting
\be X=H^1_0(\Omega)\times H^1_0(\Omega)\times V\la{X}\ee
and 
\be
E=E_1\times E_2\times E_u:(f,g,h)\in X\mapsto ((-\D_D)^{-1}R_1^\delta(f,g,h),(-\D_D)^{-1}R_2^\delta(f,g,h), A^{-1}R_u^\delta(f,g,h))\la{E}
\ee
we seek to show the existence of a weak solution $(\tilde{q}_1,\tilde{q}_2,\tilde u)=(\tilde{q}_1^\delta,\tilde{q}_2^\delta,{\tilde u}^\delta)\in X$ to (\ref{q1e}), (\ref{q2e}), (\ref{AB}) by verifying the hypotheses of Theorem \ref{fp} for the operator $E$ and showing that $E$ has a fixed point in $X$. 

First we prove that $E$ indeed maps $X$ into $X$ and does so continuously and compactly.
\begin{lemma}
The operator $E=E_1\times E_2\times E_u:X\to X$ is continuous and compact. 
\end{lemma}
\begin{proof}
We start with compactness. Since $(-\D_D)^{-1}$ (and $A^{-1}$) maps $L^\fr{3}{2}$ (resp. $(L^\fr{3}{2})^3$) continuously into $W^{2,\fr{3}{2}}\cap H_0^1$ (resp. $(W^{2,\fr{3}{2}})^3\cap V$), by Rellich's theorem, the maps $(-\D_D)^{-1}:L^\fr{3}{2}\to H^1_0$ and $A^{-1}:(L^\fr{3}{2})^3\to V$ are compact. Thus for compactness of $E$, it suffices to show that 
\be
\bal
&(f,g,h)\in X\mapsto R_i^\delta(f,g,h)\in L^\fr{3}{2}\\
&(f,g,h)\in X\mapsto R_u^\delta(f,g,h)\in (L^\fr{3}{2})^3\la{op}
\eal
\ee
are bounded. To this end, we compute
\be
\bal
\| R_1^\delta(f,g,h)\|_{L^\fr{3}{2}}\le&\|h\|_{L^6}\|\na\chi_\delta(f+\Gamma_1)\|_{L^2}\\
&+\|\na\chi_\delta(f+\Gamma)\|_{L^2}\|\na (-\D_W)^{-1}(\chi_\delta(f+\Gamma_1)-\chi_\delta(g+\Gamma_2))\|_{L^6}\\
&+\|\chi_\delta(f+\Gamma)\|_{L^6}\|(\chi_\delta(f+\Gamma_1)-\chi_\delta(g+\Gamma_2))\|_{L^2}\\
\le&C(\|h\|_V+\|f\|_{H^1}+\|g\|_{H^1}+1)(\|f\|_{H^1}+1).\la{b1}
\eal
\ee
In the last inequality we used the continuous embeddings $H^1\hookrightarrow L^6$ and the fact that $|\chi'_\delta|\le a$. Entirely similar estimates give
\be
\|R_2^\delta(f,g,h)\|_{L^\fr{3}{2}}\le C(\|h\|_V+\|f\|_{H^1}+\|g\|_{H^1}+1)(\|g\|_{H^1}+1).\la{b2}
\ee
Lastly we estimate $R_u^\delta$,
\be
\bal
\|R_u^\delta(f,g,h)\|_{L^\fr{3}{2}}\le&\|B(h,h)\|_{L^\fr{3}{2}}\\
&+\|(\chi_\delta(f+\Gamma_1)-\chi_\delta(g+\Gamma_2))\na(-\D_W)^{-1}(\chi_\delta(f+\Gamma_1)-\chi_\delta(g+\Gamma_2))\|_{L^\fr{3}{2}}\\
\le&\|h\|_{L^6}\|h\|_V\\
&+\|(\chi_\delta(f+\Gamma_1)-\chi_\delta(g+\Gamma_2))\|_{L^2}\|\na(-\D_W)^{-1}(\chi_\delta(f+\Gamma_1)-\chi_\delta(g+\Gamma_2))\|_{L^6}\\
\le&C(1+\|h\|_V^2+\|f\|_{L^2}^2+\|g\|_{L^2}^2).\la{b3}
\eal
\ee
The bounds (\ref{b1})-(\ref{b3}) show that the operators from (\ref{op}) are indeed bounded, and thus $E$ is compact. 

Continuity of $E$ follows from the fact that the components of $E$ are sums of compositions of the following continuous operations
\be
\bal
f\in H_1&\mapsto f+\Gamma_i\in H_1\\
f\in H_1&\mapsto \chi_\delta(f)\in L^4\\
f\in L^2&\mapsto \na(-\D_W)^{-1}f\in (H^1)^3\subset (L^4)^3\\
(f,g)\in L^4\times L^4&\mapsto fg\in L^2\\
f\in (L^2)^d&\mapsto(-\D_D)^{-1}\div f\in H_0^1\\
(f,g)\in V\times V&\mapsto B(f,g)\in V'\\
f\in (L^2)^d&\mapsto \mathbb{P}f\in H\\
f\in V'&\mapsto A^{-1}f\in V.
\eal
\ee
This completes the proof of the lemma.
\end{proof}

Now it remains to establish uniform a priori bounds (c.f. (\ref{hl})). We fix $\lambda\in[0,1]$ and assume that for some $({\tilde q}_1,{\tilde q}_2,\tilde{u})\in X$ we have
\be
({\tilde q}_1,{\tilde q}_2 ,\tilde u)=\lambda E({\tilde q}_1,{\tilde q}_2 ,\tilde u).
\ee
That is, for all $\psi_1,\psi_2\in H^1_0(\Omega)$ and $\psi_u\in V$ we assume we have
\begin{align}
    \int_\Omega \na \tilde{q}_i\cdot\na \psi_i\,dV&=\lambda\int_\Omega R_i^\delta({\tilde q}_1,{\tilde q}_2,\tilde{u})\psi_i\,dV,\quad i=1,2\la{w}\\
     \int_\Omega \na \tilde{u}:\na \psi_u\,dV&=\lambda\int_\Omega R_u^\delta({\tilde q}_1,{\tilde q}_2,\tilde{u})\cdot\psi_u\,dV.\la{ww}
\end{align}
We make the choice of test functions $\psi_i={\tilde q}_i$ and first estimate the resulting integral on the right hand side of (\ref{w}) for $i=1$, omitting for now the factor $\lambda$. Introducing the following primitive of $\chi_\delta$
\be
Q_\delta(y)=\int_0^y\chi_\delta(s)\,ds
\ee
we have, integrating by parts,
\be
\bal
\int_\Omega R_1^\delta({\tilde q}_1,{\tilde q}_2,\tilde{u}){\tilde q}_1\,dV=&\int_\Omega \tilde{u}\cdot\na Q_\delta ({\tilde c}_1)\,dV-\int_\Omega (\tilde{u}\cdot\na\Gamma_1)\chi_\delta({\tilde c}_1)\,dV\\
&-\int_\Omega \chi_\delta({\tilde c}_1)\na(-\D_W)^{-1}\tilde\rho\cdot\na {\tilde q}_1\,dV\\
=&I_1^{(1)}+I_2^{(1)}+I_3^{(1)}
\eal
\ee
where $\tilde{q_i}=\tilde{c_i}-\Gamma_i$ and $\tilde\rho=\chi_\delta({\tilde c}_1)-\chi_\delta({\tilde c}_2)$. Because $\tilde{u}$ is divergence-free, it follows after an integration by parts that 
\be
I_1^{(1)}=0.
\ee
Next, estimating $I_2^{(1)}$ we have, using the Poincaré inequality twice,
\be
|I_2^{(1)}|=\left|\int_\Omega (\tilde u\cdot\na\Gamma_1) \chi_\delta({\tilde c}_1)\,dV\right|\le C\|\tilde{u}\|_H\|{\tilde c}_1\|_{L^2}\le \fr{1}{2}\|\na {\tilde q}_1\|_{L^2}^2+C\|\tilde{u}\|_V^2+C.
\ee
Lastly we estimate $I^{(1)}_3$,
\be
\bal
I_3^{(1)}=&-\int_\Omega\chi_\delta({\tilde c}_1)\na(-\D_W)^{-1}\tilde\rho\cdot\na \tilde{c}_1\,dV+\int_\Omega\chi_\delta({\tilde c}_1)\na(-\D_W)^{-1}\tilde\rho\cdot\na\Gamma_1\,dV\\
=&-\int_\Omega\na Q_\delta(\tilde{c}_1)\cdot\na(-\D_W)^{-1}\tilde\rho\,dV+\int_\Omega\chi_\delta({\tilde c}_1)\na(-\D_W)^{-1}\tilde\rho\cdot\na\Gamma_1\,dV\\
=&-\int_\Omega\na (Q_\delta(\tilde{c}_1)-Q_\delta(\Gamma_1))\cdot\na(-\D_W)^{-1}\tilde\rho\,dV-\int_\Omega\na Q_\delta(\Gamma_1)\cdot\na(-\D_W)^{-1}\tilde\rho\,dV\\
&+\int_\Omega\chi_\delta({\tilde c}_1)\na(-\D_W)^{-1}\tilde\rho\cdot\na\Gamma_1\,dV\\
=&-\int_\Omega (Q_\delta (\tilde{c}_1)-Q_\delta(\Gamma_1))\tilde\rho\,dV-\int_\Omega\na Q_\delta(\Gamma_1)\cdot\na(-\D_W)^{-1}\tilde\rho\,dV\\
&+\int_\Omega\chi_\delta({\tilde c}_1)\na(-\D_W)^{-1}\tilde\rho\cdot\na\Gamma_1\,dV.
\eal
\ee
Analogous computations for $i=2$ in (\ref{w}) yield on the right hand side
\be
\int_\Omega R_2^\delta({\tilde q}_1,{\tilde q}_2,\tilde{u}){\tilde q}_2\,dV=I_1^{(2)}+I_2^{(2)}+I_3^{(2)}
\ee
where $I_j^{(2)}$, $j=1,2,3$ satisfy
\be\bal
    I_1^{(2)}=&0\\
    |I_2^{(2)}|\le&\fr{1}{2}\|\na {\tilde q}_2\|_{L^2}^2+C\|\tilde{u}\|_V^2+C\\
    I_3^{(2)}=&\int_\Omega(Q_\delta(\tilde{c}_2)-Q_\delta(\Gamma_2))\tilde\rho\,dV+\int_\Omega\na Q_\delta(\Gamma_2)\cdot\na(-\D_W)^{-1}\tilde\rho\,dV\\
    &-\int_\Omega\chi_\delta({\tilde c}_2)\na(-\D_W)^{-1}\tilde\rho\cdot\na\Gamma_2\,dV.
\eal\ee
Thus, summing (\ref{w}) in $i$ we obtain
\be
\bal
\fr{1}{2}\sum_i\|\na \tilde{q_i}\|_{L^2}^2+\lambda\int_\Omega(Q_\delta(\tilde{c}_1)-Q_\delta(\tilde{c}_2))\tilde\rho\,dV\le& C\lambda(1+ \|\tilde\rho\|_{L^1}+\|\na(-\D_W)^{-1}\tilde\rho\|_{L^1}\\
&+\|\tilde{u}\|_V^2+\sum_i\|\tilde{c}_i\|_{L^2}\|\na(-\D_W)^{-1}\tilde\rho\|_{L^2}).\la{e1}
\eal
\ee
Next, using the bounds
\be
\bal
\|\tilde{c}_i\|_{L^2}&\le\|\tilde{q}_i\|_{L^2}+ C\le C\|\na\tilde{q_i}\|_{L^2}+C\\
\|\tilde\rho\|_{L^1}&\le C\|\tilde\rho\|_{L^3}\\
\|\na(-\D_W)^{-1}\tilde\rho\|_{L^1}&\le C\|\na(-\D_W)^{-1}\tilde\rho\|_{L^2}\le C\|\tilde\rho\|_{L^3}+C
\eal
\ee
we obtain from (\ref{e1}), using Young's inequalities
\be
\fr{1}{4}\sum_i\|\na \tilde{q}_i\|_{L^2}+\lambda\int_\Omega (Q_\delta(\tilde{c}_1)-Q_\delta(\tilde{c}_2))\tilde\rho\,dV\le C + \theta\lambda\|\tilde\rho\|_{L^3}^3+C\lambda\|\tilde{u}\|_V^2\la{bla}
\ee
where $\theta$ is a small constant to be chosen later. Next we prove the following bound,
\be
\int_\Omega (Q_\delta(\tilde{c}_1)-Q_\delta(\tilde{c}_2))\tilde\rho\,dV\ge \fr{1}{4}\|\tilde\rho\|_{L^3}^3-C. \la{r3}
\ee
Prior to establishing this lower bound, we prove the following lemma, which shows that $Q_\delta\approx \fr{\chi_\delta^2}{2}$.
\begin{lemma}$|Q_\delta(y)-\fr{\chi_\delta^2}{2}(y)|\le\fr{\delta^2}{2}\,\text{for all }y\in\mathbb{R}.$\la{lem1}
\end{lemma}
\begin{proof} 
For $y\le 0$, we have $Q_\delta(y)=\chi_\delta(y)=0$ so we may assume $y>0$. Suppose $y\ge \delta$. Then
\be
Q_\delta(y)=\int_0^\delta \chi_\delta(s)\,ds+\int_\delta^y s\,ds\le \delta^2+\fr{1}{2}(y^2-\delta^2)=\fr{\delta^2}{2}+\fr{\chi^2_\delta(y)}{2}
\ee
and similarly
\be
Q_\delta(y)=\int_0^\delta \chi_\delta(s)\,ds+\int_\delta^y s\,ds\ge \fr{1}{2}(y^2-\delta^2)=-\fr{\delta^2}{2}+\fr{\chi_\delta^2(y)}{2}.
\ee
Thus the lemma holds for $y\ge \delta$. Lastly, suppose $y\in(0,\delta)$. Then, using the monotonicity of $\chi_\delta$,
\be
Q_\delta(y)=\int_0^y\chi_\delta(s)\,ds\le y\chi_\delta(y)\le\delta\chi_\delta(y)\le\fr{\delta^2}{2}+\fr{\chi^2_\delta(y)}{2}.
\ee
On the other hand, we have
\be
Q_\delta(y)\ge 0 \ge -\fr{\delta^2}{2}+\fr{\chi_\delta^2(y)}{2}.
\ee
This completes the proof of the lemma.
\end{proof}
Now we proceed with the proof of (\ref{r3}). We split $\Omega=\{\tilde\rho\ge 0\}\cup\{\tilde\rho<0\}.$ Restricted to $\{\tilde\rho\ge 0\}$, we have, using Lemma \ref{lem1},
\be
Q_\delta(\tilde{c}_1)-Q_\delta(\tilde{c}_2)\ge\fr{\chi_\delta^2(\tilde{c}_1)}{2}-\fr{\chi_\delta^2(\tilde{c}_2)}{2}-\delta^2=\fr{1}{2}(\chi_\delta({\tilde c}_1)+\chi_\delta({\tilde c}_2))\tilde\rho-\delta^2
\ee
and, restricted to $\{\tilde\rho<0\}$, again using the lemma, we have
\be
Q_\delta(\tilde{c}_1)-Q_\delta(\tilde{c}_2)\le \fr{\chi_\delta^2(\tilde{c}_1)}{2}-\fr{\chi_\delta^2(\tilde{c}_2)}{2}+\delta^2=\fr{1}{2}(\chi_\delta({\tilde c}_1)+\chi_\delta({\tilde c}_2))\tilde\rho+\delta^2.
\ee
It follows that
\be
\bal
\int_\Omega (Q_\delta(\tilde{c}_1)-Q_\delta(\tilde{c}_2))\tilde\rho\,dV&\ge \fr{1}{2}\int_\Omega\tilde\rho^2(\chi_\delta({\tilde c}_1)+\chi_\delta({\tilde c}_2))\,dV-\delta^2\|\tilde\rho\|_{L^1}\\
&\ge \fr{1}{2}\int_\Omega |\tilde\rho|^3\,dV-C\|\tilde\rho\|_{L^3}\\
&\ge \fr{1}{4}\|\tilde\rho\|_{L^3}^3-C
\eal
\ee
where in the second inequality we used the fact that, because $\chi_\delta\ge 0$, we have $|\tilde\rho|\le \chi_\delta({\tilde c}_1)+\chi_\delta({\tilde c}_2)$. Thus we have established (\ref{r3}). Now we return to (\ref{bla}) and select $\theta=1/8$. Then using the bound (\ref{r3}), we ultimately have
\be
\fr{1}{4}\sum_i\|\na \tilde{q}_i\|_{L^2}^2+\fr{\lambda}{8}\|\tilde\rho\|_{L^3}^3\le C+\tilde C\lambda\|\tilde{u}\|_V^2\la{qq}
\ee
with $\tilde C$ depending only on data and parameters. Next, we proceed to (\ref{ww}). Choosing $\psi_u=\tilde{u}$, we have on the right hand side, omitting for now the prefactor $\lambda$,
\be
\bal
\int_\Omega R_u^\delta({\tilde q}_1,{\tilde q}_2,\tilde{u})\cdot \tilde{u}\,dV&=-\int_\Omega B(\tilde{u},\tilde{u})\cdot \tilde{u}\,dV-\int_\Omega\mathbb{P}(\tilde\rho\na(-\D_W)^{-1}\tilde\rho)\cdot\tilde{u}\,dV.
\eal
\ee
On one hand, using the self-adjointness of the projection $\mathbb{P}$ and the fact that $\tilde{u}$ is divergence-free we have
\be
\int_\Omega B(\tilde{u},\tilde{u})\cdot \tilde{u}\,dV=\int_\Omega (\tilde{u}\cdot\na \tilde{u})\cdot \tilde{u}\,dV=\fr{1}{2}\int_\Omega \tilde{u}\cdot\na |\tilde{u}|^2\,dV=0.
\ee
On the other hand, again using the self-adjointness of $\mathbb{P}$, we have
\be
\int_\Omega \mathbb{P}(\tilde\rho\na(-\D_W)^{-1}\rho)\cdot \tilde{u}\,dV=\int_\Omega \tilde\rho\na(-\D_W)^{-1}\tilde\rho\cdot \tilde{u}\,dV.
\ee
Thus far, from (\ref{ww}) we have 
\be
\|\tilde{u}\|_V^2=-\lambda\int_\Omega\tilde\rho\na(-\D_W)^{-1}\tilde\rho\cdot \tilde{u}\,dV.\la{uu}
\ee
To control the integral on the right hand side, we return to (\ref{w}), taking this time the test functions $\psi_1=-\psi_2=\phi_0=(-\D_D)^{-1}\tilde\rho.$ Then, on the right hand side, we have, summing in $i$, integrating by parts, and omitting for now the prefactor $\lambda$
\be
\bal
\int_\Omega [R_1^\delta({\tilde q}_1,{\tilde q}_2,\tilde{u})-R_2^\delta({\tilde q}_1,{\tilde q}_2,\tilde{u})]\phi_0\,dV=& \int_\Omega \tilde{u} \tilde\rho\cdot\na\phi_0\,dV\\
&-\sum_i\int_\Omega\chi_\delta({\tilde c}_i)|\na\phi_0|^2\,dV-\sum_i\int_\Omega\chi_\delta(\tilde{c}_i)\na\phi_W\cdot\na\phi_0\,dV\\
\le&\int_\Omega \tilde\rho\na(-\D_W)^{-1}\tilde\rho\cdot\tilde u\,dV-\int_\Omega \tilde\rho\na\phi_W\cdot\tilde u\,dV\\
&-\fr{1}{2}\sum_i\int_\Omega\chi_\delta(\tilde{c}_i)|\na\phi_0|^2\,dV+\fr{1}{2}\int_\Omega\chi_\delta({\tilde c}_i)|\na\phi_W|^2\,dV\\
\le&C+\int_\Omega\tilde\rho\na(-\D_W)^{-1}\tilde\rho\cdot \tilde u\,dV+\fr{1}{2}\|\tilde u\|_V^2+\theta\|\tilde\rho\|_{L^3}^3\\
&-\fr{1}{2}\sum_i\int_\Omega\chi_\delta({\tilde c}_i)|\na\phi_0|^2\,dV+\theta\sum_i\|\na \tilde{q}_i\|_{L^2}^2\la{one}
\eal
\ee
where $\theta$ is a constant to be chosen. Above, we have denoted $\phi_W=(-\D_W)^{-1}\tilde\rho-\phi_0$ i.e. $\phi_W$ is the unique harmonic function on $\Omega$ whose values on the boundary are given by $W$. 

On the other hand, we bound the left hand side of (\ref{w}) in two different ways depending on whether $0\le\lambda\le\Lambda$ or $\Lambda< \lambda\le 1$, where $\Lambda$ is determined below (see (\ref{lamq})). We first consider the case $0\le\lambda\le\Lambda$. Then, we bound the left hand side of (\ref{w}) as follows,
\be
\bal
\sum_i\left|\int_\Omega \na{\tilde q}_i\cdot\na \phi_0\,dV\right|\le C+ C_q\sum_i\|\na {\tilde q}_i\|_{L^2}^2\la{two}
\eal
\ee
where $C_q>0$ depends only on data and parameters. Collecting the estimates (\ref{one}), (\ref{two}), we have from (\ref{w}),
\be
\bal
0\le\fr{\lambda}{2}\sum_i\int_\Omega\chi_\delta({\tilde c}_i)|\na\phi_0|^2\,dV\le& C+\lambda\int_\Omega\tilde\rho\na(-\D_W)^{-1}\tilde\rho\cdot\tilde u\,dV\\
&+\fr{1}{2}\|\tilde u\|_V^2+\lambda\theta\|\tilde\rho\|_{L^3}^3+(C_q+\theta)\sum_i\|\na{\tilde q}_i\|_{L^2}^2.\la{asdf}
\eal
\ee
Above, we keep track of the prefactor $\lambda$ only where needed and bound $\lambda\le 1$ if this suffices. Then, adding (\ref{asdf}) to (\ref{uu}), we obtain
\be
\fr{1}{2}\|\tilde u\|_{V}^2\le C+ \lambda\theta\|\tilde\rho\|_{L^3}^3+(C_q+\theta)\sum_i\|\na{\tilde q}_i\|_{L^2}^2.
\ee
Choosing $\theta=1$ and multiplying this last inequality by $\fr{1}{8(C_q+1)}$ and adding it to (\ref{qq}), we obtain
\be
\fr{1}{8}\sum_i\|\na{\tilde q}_i\|_{L^2}^2+\fr{1}{32(C_q+1)}\|\tilde u\|_V^2\le \tilde R\la{tilr}
\ee
for some $\tilde R$ depending on data and parameters, but not on $\lambda$ or $\delta$, provided (c.f. (\ref{qq})) 
\be
\lambda\le \Lambda=\fr{1}{32\tilde C(C_q+1)}.\la{lamq}
\ee
Now we consider the case $\Lambda<\lambda\le 1$. In this case, we estimate the left hand side of (\ref{w}) as follows,
\be
\bal
\sum_i\left|\int_\Omega \na{\tilde q}_i\cdot\na \phi_0\,dV\right|\le& C+ \theta\sum_i\|\na {\tilde q}_i\|_{L^2}^2+\Lambda\theta\|\tilde\rho\|_{L^3}^3\\
\le&C+ \theta\sum_i\|\na {\tilde q}_i\|_{L^2}^2+\lambda\theta\|\tilde\rho\|_{L^3}^3.\la{three}
\eal
\ee
Then, combining (\ref{one}), (\ref{three}), we have
\be
\bal
0\le\fr{\lambda}{2}\sum_i\int_\Omega\chi_\delta({\tilde c}_i)|\na\phi_0|^2\,dV\le& C+\lambda\int_\Omega\tilde\rho\na(-\D_W)^{-1}\tilde\rho\cdot\tilde u\,dV\\
&+\fr{1}{2}\|\tilde u\|_V^2+2\lambda\theta\|\tilde\rho\|_{L^3}^3+2\theta\sum_i\|\na{\tilde q}_i\|_{L^2}^2.\la{asdfg}
\eal
\ee
Then adding (\ref{asdfg}) to (\ref{uu}), we obtain
\be
\fr{1}{2}\|\tilde u\|_V^2\le C+2\lambda\theta\|\tilde\rho\|_{L^3}^3+2\theta\sum_i\|\na{\tilde q}_i\|_{L^2}^2.\la{qwe}
\ee
Now we multiply (\ref{qwe}) by $2(1+\tilde C)$ (c.f. (\ref{qq})) and choose $\theta$ small enough so that
\be
4(1+\tilde C)\theta\le \fr{1}{8}.
\ee
Adding the resulting inequality to (\ref{qq}), we obtain
\be
\fr{1}{8}\sum_i\|\na {\tilde q}_i\|_{L^2}^2+\|\tilde u\|_V^2\le \overline R\la{barr}
\ee
for some $\overline R$ depending on data and parameters, but not on $\lambda$ or $\delta$. The two estimates (\ref{tilr}) and (\ref{barr}) verify the hypotheses of Theorem \ref{fp}, and thus it follows that there exists a weak solution $(q_1^\delta,q_2^\delta, u^\delta)\in X$ to (\ref{q1e}), (\ref{q2e}), (\ref{AB}) satisfying
\be
\sum_i\|q_i^\delta\|_{H^1}+\|u^\delta\|_V\le R \la{R}
\ee
for some $R>0$ independent of $\delta$. Before proceeding, we establish that in fact $(q_1^\delta,q_2^\delta, u^\delta)$ is smooth and satisfies uniform $H^2$ estimates.

\begin{lemma}
If $(q_1^\delta,q_2^\delta,u^\delta)\in X$ is a weak solution to (\ref{q1e}), (\ref{q2e}), (\ref{AB}), then $(q_1^\delta,q_2^\delta,u^\delta)$ is smooth, that is, 
\be
(q_1^\delta,q_2^\delta,u^\delta)\in X^k=H^k(\Omega)\times H^k(\Omega)\times (H^k(\Omega))^3\quad\text{for all } k>0.\ee
Furthermore, there exists $C_R>0$ independent of $\delta$ so that
\be
\|q_1^\delta\|_{H^2}+\|q_2^\delta\|_{H^2}+\|u^\delta\|_{H^2}\le C_R.\la{uh2}
\ee\la{smoothlem}
\end{lemma}
\begin{proof}
First we verify (\ref{uh2}). From the estimates (\ref{b1}), (\ref{b2}), (\ref{b3}) (taking $(f,g,h)=(q_1^\delta, q_2^\delta, u^\delta)$) and the uniform bound (\ref{R}) it follows that 
\be
\bal
\sum_i\|q_i^\delta\|_{W^{2,\fr{3}{2}}}+\|u^\delta\|_{W^{2,\fr{3}{2}}}\le& C\left(\sum_i\|\D q_i^\delta\|_{L^\fr{3}{2}}+\|Au^\delta\|_{L^\fr{3}{2}}\right)\\
=&C\left(\sum_i\|R_i^\delta(q_1^\delta,q_2^\delta, u^\delta)\|_{L^\fr{3}{2}}+\|R_u^\delta(q_1^\delta,q_2^\delta, u^\delta)\|_{L^\fr{3}{2}}\right)\\
\le&C\left(1+\sum_i\|q_i^\delta\|_{H^1}^2+\|u^\delta\|_{V}^2\right)\\
\le& \bar C_R
\eal
\ee
where $\bar C_R$ is independent of $\delta$. Then, due to the embedding $W^{2,\fr{3}{2}}\hookrightarrow W^{1,3}$, it follows that
\be
\sum_i\|q_i^\delta\|_{W^{1,3}}+\|u^\delta\|_{W^{1,3}}\le \tilde C_R\la{I}
\ee
for $\tilde C_R$ independent of $\delta$. Now we estimate $R_i^\delta, R_u^\delta$ in $L^2$,
\be
\bal
\|R_i^\delta(q_1^\delta, q_2^\delta, u^\delta)\|_{L^2}\le& C(\|u^\delta\|_{L^6}\|\na c_i^\delta\|_{L^3}+\|\na c_i^\delta\|_{L^3}\|\na(-\D_W)^{-1}\rho^\delta\|_{L^6}+\|c_i^\delta\|_{L^3}\|\rho^\delta\|_{L^6})\\
\|R_u^\delta(q_1^\delta, q_2^\delta, u^\delta)\|_{L^2}\le&C(\|u^\delta\|_{L^6}\|\na u\|_{L^3}+\|\rho^\delta\|_{L^3}\|\na(-\D_W)^{-1}\rho^\delta\|_{L^6})
\eal
\ee
and since all the terms on the right hand sides are bounded independently of $\delta$ due to (\ref{I}), we find that $\D q_i^\delta$, $Au^\delta$ are bounded in $L^2$ independently of $\delta$, and (\ref{uh2}) follows. 

Higher regularity follows by induction. Indeed, suppose
\be
(q_1^\delta, q_2^\delta, u^\delta)\in X^k\text{ for some integer } k\ge 3.\la{k}
\ee
We show that $(q_1^\delta ,q_2^\delta, u^\delta)\in X^{k+1}$ follows. By elliptic regularity, it suffices to show that $R_i^\delta(q_1^\delta,q_2^\delta, u^\delta)$ and $R_u^\delta(q_1^\delta,q_2^\delta, u^\delta)$ are in $H^{k-1}$, and so we compute
\be
\bal
\|R_i^\delta(q_1^\delta, q_2^\delta, u^\delta)\|_{H^{k-1}}\le& C(\|u^\delta\cdot\na c_i^\delta\|_{H^{k-1}}+\|\na c_i^\delta\cdot\na(-\D_W)^{-1}\rho^\delta\|_{H^{k-1}}+\|c_i^\delta\rho^\delta\|_{H^{k-1}})\\
\le& C(\|u^\delta\|_{H^{k-1}}\|\na c_i^\delta\|_{H^{k-1}}+\|\na c_i^\delta\|_{H^{k-1}}\|\na(-\D_W)^{-1}\rho^\delta\|_{H^{k-1}}\\
&+\|c_i^\delta\|_{H^{k-1}}\|\rho^\delta\|_{H^{k-1}})\\
\le& C(\|u^\delta\|_{H^{k-1}}\|c_i^\delta\|_{H^k}+\|c_i^\delta\|_{H^k}(1+\|\rho^\delta\|_{H^{k-2}})+\|c_i^\delta\|_{H^{k-1}}\|\rho^\delta\|_{H^{k-1}})\\
<&\infty\quad\text{by }(\ref{k})
\eal
\ee
where in the second inequality, we used the fact that $H^s$ is an algebra for $s> \fr{3}{2}$; that is,
$$\|fg\|_{H^s}\le C\|f\|_{H^s}\|g\|_{H^s}.$$
Similarly, we estimate
\be
\bal
\|R_u^\delta(q_1^\delta,q_2^\delta, u^\delta)\|_{H^{k-1}}\le&C(\|u^\delta\cdot\na u^\delta\|_{H^{k-1}}+\|\rho^\delta\na(-\D_W)^{-1}\rho^\delta\|_{H^{k-1}})\\
\le& C\|u^\delta\|_{H^{k-1}}\|u^\delta\|_{H^k}+\|\rho^\delta\|_{H^{k-1}}(1+\|\rho^\delta\|_{H^{k-2}})\\
<&\infty\quad\text{by }(\ref{k})
\eal
\ee
where in the first inequality, we used the fact that $\mathbb{P}$ is continuous as a mapping $\mathbb{P}:H^{k-1}\to H^{k-1}$ \cite{temam}.

The proof of the smoothness of $(q_1^\delta,q_2^\delta, u^\delta)$ is thus complete once we verify the base case of $k=3$. It suffices to show that $R_i^\delta(q_1^\delta, q_2^\delta,u^\delta)$ and $R_u^\delta(q_1^\delta,q_2^\delta, u^\delta)$ are in $H^1$:
\be
\bal
\|R_i^\delta(q_1^\delta,q_2^\delta, u^\delta)\|_{H^1}\le&C(\|u^\delta\cdot\na c_i^\delta\|_{H^1}+\|\div(\chi_\delta(c_i^\delta)\na(-\D_W)^{-1}\rho^\delta)\|_{H^1})\\
\le& C(\|u^\delta\cdot \na c_i^\delta\|_{L^2}+\|\na(u^\delta\cdot\na c_i^\delta)\|_{L^2}+\|\chi_\delta(c_i^\delta)\na(-\D_W)^{-1}\rho^\delta\|_{H^2})\\
\le& C(\|u^\delta\|_{L^\infty}\|\na c_i^\delta\|_{L^2}+\|\na u^\delta\|_{L^4}\|\na c_i^\delta\|_{L^4}+\|u^\delta\|_{L^\infty}\|\na\na c_i^\delta\|_{L^2}\\
&+\|\chi_\delta(c_i^\delta)\|_{H^2}\|\na(-\D_W)^{-1}\rho^\delta\|_{H^2})\\
<&\infty\quad\text{by }(\ref{uh2})\\
\|R_u^\delta(q_1^\delta,q_2^\delta, u^\delta)\|_{H^1}\le&C\|u^\delta\cdot\na u^\delta\|_{H^1}+\|\rho^\delta\na(-\D_W)^{-1}\rho^\delta\|_{H^1})\\
\le& C(\|u^\delta\cdot\na u^\delta\|_{L^2}+\|\na(u^\delta\cdot\na u^\delta)\|_{L^2}\\
&+\|\rho^\delta\na(-\D_W)^{-1}\rho^\delta\|_{L^2}+\|\na(\rho^\delta\na(-\D_W)^{-1}\rho^\delta\|_{L^2})\\
\le&C(\|u^\delta\|_{L^\infty}\|u^\delta\|_V+\|\na u^\delta\|_{L^4}^2+\|u^\delta\|_{L^\infty}\|\na\na u^\delta\|_{L^2}\\
&+\|\rho^\delta\|_{L^4}\|\na(-\D_W)^{-1}\rho^\delta\|_{L^4}+\|\na\rho^\delta\|_{L^2}\|\na(-\D_W)^{-1}\rho^\delta\|_{L^\infty}\\
&+\|\rho^\delta\|_{L^4}\|\na\na(-\D_W)^{-1}\rho^\delta\|_{L^4})\\
<&\infty\quad\text{by }(\ref{uh2})
\eal
\ee
where again, we used the fact that $\mathbb{P}:H^1\to H^1$ is continuous. Thus the proof of the lemma is complete.
\end{proof} 
Now we finish the proof of Theorem \ref{thm1}. First, we establish the nonnegativity of $c_i^\delta$. We recall that $c_i^\delta$ satisfies
\be
-\D c_i^\delta=-(u^\delta\cdot\na c_i^\delta)\chi'_\delta(c_i^\delta)+z_i(\na c_i^\delta\cdot\na\Phi^\delta)\chi'_\delta(c_i^\delta)-\chi_\delta(c_i^\delta)\rho^\delta
\ee
where $z_1=1=-z_2$. Suppose $c_i^\delta$ attains a negative value in $\Omega$, and suppose that at $x_0\in\Omega$ we have $c_i^\delta(x_0)=\inf_\Omega c_i^\delta<0$. Then consider the largest ball $B$ centered at $x_0$ so that ${c_i^\delta}_{|B}\le 0$. Since ${c_i^\delta}_{|\pa\Omega}=\gamma_i>0$, we necessarily have $\bar B\subset \Omega$ and for some $y\in\pa B$ we have $c_i^\delta(y)=0$. Furthermore, $c_i^\delta$ satisfies
\be
-\D c_i^\delta=-(u^\delta\cdot\na c_i^\delta)\chi'_\delta(c_i^\delta)+z_i(\na c_i^\delta\cdot\na\Phi^\delta)\chi'_\delta(c_i^\delta).
\ee 
in $B$. However, since ${c_i^\delta}_{|B}$ attains its global minimum in $B$, the strong maximum principle implies that ${c_i^\delta}_{|B}\equiv\inf_\Omega c_i^\delta<0$; however this contradicts the fact that $c_i^\delta(y)=0$ for $y\in\pa B$. Therefore $c_i^\delta\ge 0$.

Now, due to (\ref{uh2}), there is a sequence $\delta_j\to 0$ as $j\to \infty$ and $(c_1,c_2,u)\in X^2$ so that $(c_1^{\delta_j},c_2^{\delta_j},u^{\delta_j})\to (c_1, c_2, u)$ strongly in $X^1$, pointwise almost everywhere, and weakly in $X^2$ as $j\to \infty$. And, we take 
\be 
\Phi=(-\D_W)^{-1}\rho.\la{pp}
\ee
Since $c_i$ is the pointwise almost everywhere limit of nonnegative functions, we have $c_i\ge 0$ almost everywhere, and after redefining $c_i$ on a set of measure zero, we assume henceforth that $c_i\ge 0$ everywhere. 

Now we verify that $(c_1, c_2, u)$ together with (\ref{pp}) is a weak solution of (\ref{s})-(\ref{B}). Since the trace operator is continuous from $H^1(\Omega)$ into $H^\fr{1}{2}(\pa\Omega)$ and $(c_1,c_2,u)$ is the strong $X^1$ limit of $(c_1^{\delta_j},c_2^{\delta_j},u^{\delta_j})$, we have that the boundary conditions (\ref{b})-(\ref{B}) are satisfied in the sense of traces. Next, we verify that $(c_1,c_2,u)$ satisfies (\ref{s})-(\ref{S}) in the weak sense: for any $(\psi_1,\psi_2,\psi_u)\in X$, we have
\be
\bal
\int_\Omega \na c_i\cdot\na\psi_i\,dV&=-\int_\Omega(-uc_i+c_i\na(-\D_W)^{-1}\rho)\cdot\na\psi_i\,dV\\
\int_\Omega \na u:\na\psi_u\,dV&=-\int_\Omega (B(u,u)+\mathbb{P}(\rho\na(-\D_W)^{-1}\rho))\cdot\psi_u\,dV.\la{wf}
\eal
\ee
Prior to establishing these equalities, we show that 
\be 
\|\chi_{\delta_j}(c_i^{\delta_j})-c_i\|_{L^6}\to 0 \quad\text{as } j\to \infty.\la{l6}
\ee
Indeed, we have
\be
\bal
\int_\Omega |\chi_{\delta_j}(c_i^{\delta_j})-c_i|^6\,dV\le&C\left(\int_\Omega|\chi_{\delta_j}(c_i^{\delta_j})-\chi_{\delta_j}(c_i)|^6\,dV+\int_\Omega |\chi_{\delta_j}(c_i)-c_i|^6\,dV\right).
\eal
\ee
The first integral on the right hand side converges to $0$ because
\be
\|\chi_{\delta_j}(c_i^{\delta_j})-\chi_{\delta_j}(c_i)\|_{L^6}\le a\|c_i^{\delta_j}-c_i\|_{L^6}\le C\|c_i^{\delta_j}-c_i\|_{H^1}\to 0.
\ee
The second integral also converges to $0$ due to the dominated convergence theorem and the fact that for each $x\in\Omega$, we have $\chi_{\delta_j}(c_i(x))=c_i(x)$ for all $j$ sufficiently large since $\chi_\delta(y)=y$ for all $\delta\le y$ if $y>0$ and for all $\delta$ if $y=0$. 
Thus, we can now compute
\begin{align*}
\left|\int_\Omega \na c_i\cdot\na\psi_i\,dV-\int_\Omega\na c_i^{\delta_j}\cdot\na\psi_i\,dV\right|\le&\|\psi_i\|_{H^1}\|c_i-c_i^{\delta_j}\|_{H^1}\to 0\\\\
\left|\int_\Omega uc_i\cdot\na\psi_i\,dV-\int_\Omega u^{\delta_j} \chi_{\delta_j}(c_i^{\delta_j})\cdot\na\psi_i\,dV\right|\le&\|\psi_i\|_{H^1}(\|u\|_{L^3}\|c_i-\chi_{\delta_j}(c_i^{\delta_j})\|_{L^6}\\
&+\|\chi_{\delta_j}(c_i^{\delta_j})\|_{L^3}\|u-u^{\delta_j}\|_{L^6})\to 0
\end{align*}

\begin{align*}
&\left|\int_\Omega c_i\na(-\D_W)^{-1}\rho\cdot\na\psi_i\,dV-\int_\Omega\chi_{\delta_j}(c_i^{\delta_j})\na(-\D_W)^{-1}\rho^{\delta_j}\cdot\na\psi_i\,dV\right|\\
\le&\|\psi_i\|_{H^1}(\|\na(-\D_W)^{-1}\rho\|_{L^3}\|c_i-\chi_{\delta_j}(c_i^{\delta_j})\|_{L^6}+\|\na(-\D_D)^{-1}(\rho-\rho^{\delta_j})\|_{L^6}\|\chi_{\delta_j}(c_i^{\delta_j})\|_{L^3})\to 0
\end{align*}

\begin{align*}
    \left|\int_\Omega \na u:\na\psi_u\,dV-\int_\Omega\na u^{\delta_j}:\na\psi_u\,dV\right|\le \|\psi_u\|_V\|u-u^{\delta_j}\|_V\to 0
\end{align*}

\begin{align*}
    \left|\int_\Omega B(u,u)\cdot\psi_u\,dV-\int_\Omega B(u^{\delta_j},u^{\delta_j})\cdot\psi_u\right|\le\|\psi_u\|_{L^6}(\|u-u^{\delta_j}\|_{L^3}\|u\|_V+\|u^{\delta_j}\|_{L^3}\|u-u^{\delta_j}\|_V)\to 0
\end{align*}

\begin{align*}
    &\left|\int_\Omega \rho\na(-\D_W)^{-1}\rho\cdot\psi_u\,dV-\int_\Omega\rho^{\delta_j}\na(-\D_W)^{-1}\rho^{\delta_j}\cdot\psi_u\right|\\
    \le& \|\psi\|_{L^6}(\|\rho-\rho^{\delta_j}\|_{L^2}\|\na(-\D_W)^{-1}\rho\|_{L^3}+\|\rho^{\delta_j}\|_{L^2}\|\na(-\D_D)^{-1}(\rho-\rho^{\delta_j})\|_{L^3})\to 0.
\end{align*}
The above computations, together with the fact that $(c_1^\delta,c_2^\delta,u^\delta)$ satisfy (\ref{ap})-(\ref{Sd}), imply (\ref{wf}). 

Finally, the smoothness of $(c_1,c_2,u)$ follows from the same bootstrapping scheme as in the proof of Lemma \ref{smoothlem}. The proof of Theorem \ref{thm1} is now complete.
\end{proof}

The equilibria of the Nernst-Planck-Navier-Stokes system are unique minimizers of a total energy that is nonincreasing in time on solutions \cite{ci} and they arise when certain equilibrium boundary conditions  \eqref{eqc1}, \eqref{eqc2} are supplied.  The potential then obeys  Poisson-Boltzmann equations \eqref{PB}  which provide the unique steady state solution ($c_1^*, c_2^*)$ of the Nernst-Planck equations \eqref{np} with zero fluid velocity $u^*\equiv 0$. However, in many cases of physical interest, the boundary conditions are not suitable for equilibrium, and an electrical potential gradient generates (experimentally or numerically)  nontrivial fluid flow. Thus, it is relevant to derive  conditions under which the steady state whose existence is guaranteed by Theorem \ref{thm1} has nonzero fluid velocity $u^*\not\equiv 0$. We derive below one such condition.

\begin{thm}
Suppose $(c_1^*, c_2^*, u^*)$ is a solution to (\ref{s})-(\ref{S}) with boundary conditions (\ref{b})-(\ref{B}). Suppose in addition that the boundary conditions satisfy
\be
\int_{\pa\Omega}(\gamma_1-\gamma_2)(n_i\pa_j-n_j\pa_i)W\,dS\neq 0 \quad\text{or}\quad \int_{\pa\Omega}W(n_i\pa_j-n_j\pa_i)(\gamma_1-\gamma_2)\,dS\neq 0.\la{bct}
\ee
for some $i,j\in\{x,y,z\}, i\neq j$, where $n_i$ are the components of the unit normal vector along $\pa\Omega$. Then $u^*\not \equiv 0.$\la{unot0}
\end{thm}
\begin{rem}
We note that if $i,j\in\{x,y,z\}$ with $i\neq j$, then $n_i\pa_j-n_j\pa_i$ is a vector field tangent to $\pa\Omega$, so that the integrals in (\ref{bct}) are well defined. Indeed, the characteristic directions of $n_i\pa_j-n_j\pa_i$ are $n_ie_j-n_je_i$ (with $e_k$ the canonical basis of $\Rr^d$) and $n\cdot (n_ie_j-n_je_i) = 0$ shows that these are tangent to $\pa\Omega$. Thus, the condition (\ref{bct}) can be checked with just knowledge of the values of $c_i$ and $\Phi$ on $\pa\Omega$.
\end{rem}
\begin{proof}
If $(c_1^*, c_2^*, u^*\equiv 0)$ is a solution to (\ref{s})-(\ref{S}), then $\rho^*\na\Phi$ must be a gradient force i.e.
\be
\rho^*\na\Phi^*=\na F
\ee
for some smooth $F$. Thus a sufficient condition for $u^*\not\equiv 0$ is 
\be\na\times(\rho^*\na\Phi^*)\not\equiv 0.\la{curl}\ee
In turn, a sufficient condition for (\ref{curl}) is
\be
\int_\Omega \na\rho^*\times\na\Phi^*\,dV\neq 0.\la{curll}
\ee
Integrating the above integral by parts, moving the derivative off $\rho^*$, we obtain the following equivalent condition
\be
\int_{\pa\Omega}\rho^*(n_i\pa_j-n_j\pa_i)\Phi\,dS\neq 0\quad \text{for some $i,j\in\{x,y,z\},$ $i\neq j.$}
\ee
And by the remark following the statement of Theorem \ref{unot0}, this is equivalent to the condition
\be
\int_{\pa\Omega}(\gamma_1-\gamma_2)(n_i\pa_j-n_j\pa_i)W\,dS\neq 0\quad \text{for some $i,j\in\{x,y,z\},$ $i\neq j.$}
\ee
Similarly, by moving the derivative off $\Phi^*$ in (\ref{curll}), we obtain the equivalent condition
\be
\int_{\pa\Omega}W(n_i\pa_j-n_j\pa_i)(\gamma_1-\gamma_2)\,dS\neq 0\quad \text{for some $i,j\in\{x,y,z\},$ $i\neq j$.}
\ee
This completes the proof.
\end{proof}

\section{Maximum principle and long time behavior of solutions}\la{lt}
In this section we investigate the long time behavior of the system 
\begin{align}
\pa_t c_1+u\cdot\na c_1=&D_1\div(\na c_1+c_1\na\Phi)\la{tnp1}\\
\pa_t c_2+u\cdot\na c_2=&D_2\div(\na c_2-c_2\na\Phi)\la{tnp2}\\
-\epsilon\D\Phi&=\rho=c_1-c_2\la{tpois}\\
\pa_tu-\nu\D u+\na p&=-K\rho\na\Phi\la{tstokes}\\
\div u&=0.\la{tdiv}
\end{align}

The global existence and uniqueness of smooths solutions of this system with Dirichlet boundary conditions is proved in \cite{cil}. Here we prove a maximum/minimum principle for the ionic concentrations $c_i$, which in particular gives us time independent $L^\infty$ bounds (see also \cite{EN}). In addition, we show that the Dirichlet boundary data of $c_i$ are \textit{attracting} in the sense that $\max\{\sup_\Omega c_1,\sup_\Omega c_2\}$ and $\min\{\sup_\Omega c_1,\sup_\Omega c_2\}$ converge monotonically to the extremal values of the boundary values in the limit of $t\to\infty$. 

The restriction to the Stokes subsystem is due to lack of information on global regularity for Navier-Stokes solutions in 3D, thus limiting the analysis of long time behavior. The results  below do extend to 2D NPNS and apply  to 3D NPNS under the assumption of regularity of velocity. The modifications to the proofs required in these cases are straightforward but will not be pursued here.

We consider a general smooth, bounded, connected domain $\Omega\subset\mathbb{R}^3$ with boundary conditions
\be
\bal
{c_i}_{|\pa\Omega}&=\gamma_i>0\\
\Phi_{|\pa\Omega}&=W\\
u_{|\pa\Omega}&=0\la{bbcc}
\eal
\ee
where $\gamma_i$ and $W$ are smooth and not necessarily constant. \begin{thm}
Suppose $(c_1\ge 0,c_2\ge 0,u)$ is the unique, global smooth solution to (\ref{tnp1})-(\ref{tdiv}) on $\Omega$ with smooth initial conditions $(c_1(0)\ge 0,c_2(0)\ge 0,u(0))$ (with $\div u(0)=0$) and boundary conditions (\ref{bbcc}). Then,
\begin{enumerate}[(I)]
    \item For $i=1,2$ and all $t\ge 0$ \begin{align}\min\{\inf_\Omega c_1(0), \inf_\Omega c_2(0),\ug\}\le c_i(t,x)\le \max\{\sup_\Omega c_1(0), \sup_\Omega c_2(0),\og\}\end{align}
    where $\ug=\min_i\inf_{\pa\Omega}\gamma_i$ and $\og=\max_i\sup_{\pa\Omega}\gamma_i$. In particular 
    \begin{align}\om(t)=\max_i\sup_\Omega c_i(t,x),\quad \um(t)=\min_i\inf_\Omega c_i(t,x)\end{align}
    are nonincreasing and nondecreasing on $(0,\infty)$, respectively.
    \item For all $\delta>0$, there exists $T$ depending on $\delta$, $\Omega$ and initial and boundary conditions such that for all $t\ge T$ we have 
    $$\ug-\delta\le \um(t)\le\om(t)\le\og+\delta.$$
\end{enumerate}\la{maxthm}
\end{thm}

The theorem is a consequence of the following proposition.

\begin{prop}
Suppose $v_i:[0,\infty)\times\bar\Omega\to \mathbb{R}$, $i=1,2$ is a nonnegative, smooth solution to 
\be
\bal
\pa_t v_1&=d_1\D v_1+b_1\cdot\na v_1-p_1(v_1-v_2)\\
\pa_t v_2&=d_2\D v_2+b_2\cdot\na v_2+p_2(v_1-v_2)
\eal
\ee
with time independent, smooth Dirichlet boundary conditions
\be
{v_i}_{|\pa\Omega}=g_i>0
\ee
where $d_i>0$ are constants, $b_i=b_i(t,x)$ are smooth vector fields, and $p_i=p_i(t,x)\ge 0$. Then
\begin{enumerate}[(I')]
    \item For both $i$ and all $t\ge 0$ \begin{align}\min\{\inf_\Omega v_1(0), \inf_\Omega v_2(0),\ugg\}\le v_i(t,x)\le \max\{\sup_\Omega v_1(0), \sup_\Omega v_2(0),\ogg\}\la{lub}\end{align}
    where $\ugg=\min_i\inf_{\pa\Omega}g_i$ and $\ogg=\max_i\sup_{\pa\Omega}g_i$. In particular 
    \begin{align}\oV(t)=\max_i\sup_\Omega v_i(t,x),\quad \uV(t)=\min_i\inf_\Omega v_i(t,x)\la{umom}\end{align}
    are nonincreasing and nondecreasing on $(0,\infty)$, respectively.
    \item Suppose, in addition to the preceding hypotheses, that $b_i$ is uniformly bounded in time. Then for all $\delta>0$, there exists $0<T^*=T^*(\delta, d_i,\sup_t\|b_i(t)\|_{L^\infty},g_i, v_i(t=0),\Omega)$ such that for all $t\ge T^*$ we have 
    $$\ugg-\delta\le \uV(t)\le\oV(t)\le\ogg+\delta.$$
\end{enumerate}\la{prop!}
\end{prop}
\begin{proof}
We prove just the lower bound in (I') as the upper bound can be established analogously. If either $\inf_\Omega v_1(t=0)=0$ or $\inf_\Omega v_2(t=0)=0$ then the lower bound holds trivially as we are assuming that $v_i\ge 0$. So we assume $v_1(t=0),v_2(t=0)>0$.

We define $$\uuV(t)=\min_{0\le s\le t}\uV(s).$$ We show that $\uV$ and $\uuV$ are both locally Lipschitz (i.e. Lipschitz continuous on every interval $[0,T]$). Indeed, assigning to each $t\ge 0$ a point $x_i(t)\in\bar{\Omega}$ such that $v_i(t,x_i(t))=\uV_i(t)=\inf_\Omega v_i(t,x)$, we have for $s<t\le T$
\be
\fr{v_i(t,x_i(t))-v_i(s,x_i(t))}{t-s}\le \fr{\uV_i(t)-\uV_i(s)}{t-s}\le \fr{v_i(t,x_i(s))-v_i(s,x_i(s))}{t-s}
\ee
and so
\be
\left|\fr{\uV_i(t)-\uV_i(s)}{t-s}\right|\le \sup_{[0,T]\times\bar\Omega}|\pa_tv_i(t,x)|=\underline L^T_i
\ee
implying that $\uV_i(t)$ is locally Lipschitz. Next, assigning to each $t\ge 0$ an $i(t)\in\{1,2\}$ such that $\uV_{i(t)}(t)=\uV(t)$ we have for $s<t\le T$
\be
\fr{\uV_{i(t)}(t)-\uV_{i(t)}(s)}{t-s}\le \fr{\uV(t)-\uV(s)}{t-s}\le\fr{\uV_{i(s)}(t)-\uV_{i(s)}(s)}{t-s}
\ee
and thus
\be
\left|\fr{\uV(t)-\uV(s)}{t-s}\right|\le \max_i \underline L^T_i=\underline L^T.
\ee
Thus $\uV_i$ is locally Lipschitz. Lastly consider $\uuV$. Fixing $s<t\le T$, we assume without loss of generality that $\uuV(t)\neq \uuV(s)$. In particular since $\uuV$ is nonincreasing, this implies $\uuV(t)<\uuV(s)$. Then consider
\be
t^*=\inf\{t'\in[s,t] |\,\uV(t')=\uuV(t)\}.
\ee
Since $\uV(s)\ge\uuV(s)>\uuV(t)$, we necessarily have $t^*>s.$ Thus,
\be
\left|\fr{\uuV(t)-\uuV(s)}{t-s}\right|=\fr{\uuV(s)-\uuV(t)}{t-s}\le \fr{\uV(s)-\uV(t^*)}{t-s}\le \fr{\uV(s)-\uV(t^*)}{t^*-s}\le\underline L^T
\ee
and local Lipschitz continuity of $\uuV$ follows.

Due to the Lipschitz continuity, we have in particular that $\uV_i,\uV,\uuV$ are differentiable almost everywhere and the set
$A_T=\{t\in (0,T)|\,\uV_1'(t),\uV_2'(t),\uV'(t),\uuV'(t) \text{ exist}\}$ has full measure, $|A_T|=T$. To complete the proof of the lower bound in (I'), we prove the following lemma.
\begin{lemma}
For all $t\in A_T$ we have
\begin{enumerate}[(i)]
    \item $\uV_i'(t)=\pa_tv_i(t,x)$ for all $x\in\bar\Omega$ such that $v_i(t,x)=\uV_i(t)$
    \item $\uV'(t)=\uV_i'(t)$ for all $i$ such that $\uV(t)=\uV_i(t)$
    \item if $\uuV'(t)<0$, then $\uuV(t)=\uV(t)$ and $\uuV'(t)=\uV'(t).$
\end{enumerate}\la{mlem}
\end{lemma}
\begin{proof}
To see (i), we fix $x\in\bar\Omega$ such that $v_i(t,x)=\uV_i(t)$ and compute for $0<s<t<T$
\be
\fr{\uV_i(t)-\uV_i(s)}{t-s}\ge\fr{v_i(t,x)-v_i(s,x)}{t-s}
\ee
and taking the limit $s\to t^-$ we see that $\uV_i'(t)\ge \pa_tv_i(t,x)$. Similarly for $0<t<s<T$ we have
\be
\fr{\uV_i(s)-\uV_i(t)}{s-t}\le\fr{v_i(s,x)-v_i(t,x)}{s-t}
\ee
and we obtain $\uV_i'(t)\le \pa_tv_i(t,x)$ upon taking the limit $s\to t^+$. An analogous argument gives us (ii).

Now we show (iii). Assume $\uuV'(t)<0$, then for all $s<t$ we have $\uuV(s)>\uuV(t)$ for otherwise, since $\uuV$ is nonincreasing we have $\uuV(s)=\uuV(t)$ for some $s<t$. But then, again since $\uuV$ is nonincreasing we have $\uuV(r)=\uuV(t)$ for all $r\in[s,t]$. However, it then follows that the left sided derivative of $\uuV$ at $t$ is zero, which contradicts our assumption $\uuV'(t)<0$. By the same argument we see that for all $s>t$ we have $\uuV(s)<\uuV(t)$.

It now follows that $\uuV(t)=\uV(t)$. Indeed, otherwise, there exists $s<t$ such that $\uV(s)=\uuV(t).$ But by the argument in the previous paragraph, we have $\uV(s)\ge \uuV(s)>\uuV(t)$, which gives us a contradiction.

Now to prove that $\uuV'(t)=\uV'(t)$, we compute for $s<t$, using $\uuV(t)=\uV(t)$
\be
\fr{\uuV(t)-\uuV(s)}{t-s}\ge \fr{\uV(t)-\uV(s)}{t-s}
\ee
which gives us $\uuV'(t)\ge \uV'(t)$ upon taking the limit $s\to t^-$. Similarly, considering $s>t$ we obtain the opposite inequality, thus completing the proof of (iii).\end{proof}

Now we complete the proof of the lower bound in (I'). Suppose for the sake of contradiction that there exists $T>0$ such that $\uuV(T)<\min\{\inf_\Omega v_1(0),\inf_\Omega v_2(0),\ugg\}.$ Then since $\uuV$ is locally Lipschitz and hence satisfies the fundamental theorem of calculus, there exists $t\in A^T$ such that $\uuV'(t)<0$ and $\uuV(t)<\min\{\inf_\Omega v_1(0),\inf_\Omega v_2(0),\ugg\}$. Then it follows from Lemma \ref{mlem} that for some $i$ and some $x\in\bar\Omega$, we have $v_i(t,x)=\uuV(t)$ and $\pa_tv_i(t,x)<0$. Also, since $\uuV(t)<\ugg$ we have that $x\in\Omega$. We assume without loss of generality that $i=1$. Then evaluating (\ref{tnp1}) at $(t,x)$ and using the fact that $v_2(t,x)\ge \uuV(t)$ we find that
\be
\bal
0>\pa_tv_1=&d_1\D v_1+b_1\cdot\na v_1-p_1(v_1-v_2)\\
\ge& -p_1(\uuV(t)-\uuV(t))\\
=& 0
\eal
\ee
which gives us a contradiction. Therefore $\uV(t)\ge\uuV(t)\ge \min\{\inf_\Omega v_1(0),\inf_\Omega v_2(0),\ugg\}$ for all $t$. The monotonicity of $\uV$ follows from the same argument by replacing the initial time $0$ by some arbitrary time $s>0$. Then we obtain for all $t>s$
\be
\uV(t)\ge \uuV(t)\ge \min\{\inf_\Omega v_1(s),\inf_\Omega v_2(s)\}=\uV(s).
\ee
Above we removed the $\ug$ from the minimum, as this is redundant for $s>0$.

Now we prove the lower bound statement of (II'). The upper bound statement is proved similarly.

We assume without loss of generality that $0\not\in\bar\Omega$ so that there exists $\underline\alpha,\overline\alpha>0$ such that $\underline\alpha\le |x|\le\overline\alpha$ for all $x\in\bar\Omega$. Then we define
\be
w_i=v_i-\epsilon|x|^\lambda
\ee
where $\lambda>0$ is chosen large enough so that for each $i$ we have
\be
\fr{d_i}{2}(\lambda+1)\ge \sup_{t\ge 0,\,x\in\Omega}|b_i(t,x)\cdot x| .\la{lbig}
\ee
and $\epsilon>0$ is chosen small enough that
\be
\epsilon\overline\alpha^\lambda\le \delta.\la{epsmall}
\ee
The functions $w_i$ satisfy the equations
\be
\pa_t w_i=d_i\D w_i+b_i\cdot\na w_i+\epsilon d_i\lambda(\lambda+1)|x|^{\lambda-2}+\epsilon\lambda|x|^{\lambda-2}b_i\cdot x-z_ip_i(w_1-w_2)\la{Wi}
\ee
where $z_1=1=-z_2$. 

By the same proof as in Lemma \ref{mlem} and the discussion leading up to it, we know that the functions
\be
\uW_i(t)=\inf_\Omega w_i(t,x),\quad \uW(t)=\min_i \uW_i(t)
\ee
are locally Lipschitz and thus differentiable almost everywhere. In addition, for each $t>0$ where $\uW'(t)$, $\uW_1'(t)$, $\uW_2'(t)$ all exist, we have that for each $i$ and $x\in\bar\Omega$ such that $w_i(x,t)=\uW(t)$, the time derivatives coincide, $\pa_t w_i(x,t)=\uW'(t)$. 

Now, if initially, at time $t=0$, we have $\ugg-\epsilon\sup_{\pa\Omega}|x|^\lambda\le \uW(0)$, then since $\uW\le \uV$ we have, using (\ref{epsmall}),
\be
\ugg-\delta\le\ugg-\epsilon\overline\alpha^\lambda\le \uV(0).
\ee
Then, since $\uV$ is monotone nondecreasing, the lower bound in (II') follows. 

Now suppose $\uW(0)<\ugg-\epsilon\sup_{\pa\Omega}|x|^\lambda.$ Then in particular, we have
\be
\uW(0)<\min_i\inf_{\pa\Omega}w_i=\min_i\inf_{\pa\Omega}(g_i-\epsilon|x|^\lambda).
\ee
Thus by continuity, we have that
\be
\uW(t)<\min_i\inf_{\pa\Omega}(g_i-\epsilon|x|^\lambda)\la{www}
\ee
holds on some interval $[0,T^*)$ where $T^*\in (0,\infty]$ can be chosen to be maximal so that if $T^*$ is finite, we have  $\uW(T^*)=\min_i\inf_{\pa\Omega}(g_i-\epsilon|x|^\lambda)$. 

We claim that indeed $T^*<\infty$. The lower bound of (II') follows from this claim since
\be
\bal
\uW(T^*)=\min_i\inf_{\pa\Omega}(g_i-\epsilon|x|^\lambda)\Rightarrow& \min_i\inf_\Omega (v_i(T^*)-\epsilon|x|^\lambda)\ge \ugg-\epsilon\sup_{\pa\Omega}|x|^\lambda\ge \ugg -\delta\\
\Rightarrow& \uV(T^*)\ge \ugg-\delta
\eal
\ee
and the lower bound continues to hold for all $t\ge T^*$ due to the monotonicity of $\uV$. 

It remains to prove the claim that $T^*<\infty$. Indeed, let us fix a time $t$ such that (\ref{www}) holds. Then at time $t$, the value $\uW(t)$ is attained by $w_i$, for some $i$, at some interior point $x\in\Omega$. This point is a global minimum of $w_i$ at time $t$. We assume without loss of generality that $i=1$. Thus evaluating (\ref{Wi}) at $(t,x)$ and using (\ref{lbig}) and the fact that $w_2(t,x)\ge w_1(t,x)$ we have
\be
\bal
\pa_t w_1=&d_1\D w_1+b_1\cdot\na w_1+\epsilon d_1\lambda(\lambda+1)|x|^{\lambda-2}+\epsilon\lambda|x|^{\lambda-2}b_1\cdot x-p_1(w_1-w_2)\\
\ge& \epsilon d_1\lambda(\lambda+1)|x|^{\lambda-2}-\epsilon\lambda|x|^{\lambda-2}\sup_{y\in\Omega}|b_1(t,y)\cdot y|\\
\ge&\fr{1}{2} \epsilon d_1\lambda(\lambda+1)|x|^{\lambda-2}\\
\ge&\fr{1}{2}\epsilon d_1\lambda(\lambda+1)\underline\alpha^{\lambda-2}.
\eal
\ee
If $t\in A_{T^*}=\{t\in (0,T^*)| \uW'(t), \uW_1'(t), \uW_2'(t)\,\text{exist}\}$, then by the same argument as in Lemma \ref{mlem}, we have  $\uW'(t)=\pa_tw_1(t,x)$ and thus
\be
\uW'(t)\ge \fr{1}{2}\epsilon d_1\lambda(\lambda+1)\underline\alpha^{\lambda-2}.
\ee
In general, the relation
\be
\uW'(t)\ge \min_i\fr{1}{2}\epsilon d_i\lambda(\lambda+1)\underline\alpha^{\lambda-2}=\tilde\beta.
\ee
holds for every time $t\in A_{T^*}$. Since $\uW, \uW_1, \uW_2$ are each locally Lipschitz (and hence differentiable almost everwhere and satisfy the fundamental theorem of calculus), if (\ref{www}) holds on $[0,\infty)$ (i.e. if $T^*=\infty$), then we obtain
\be
\infty> \liminf_{t\to\infty}\uW(t)=\uW(0)+\liminf_{t\to\infty}\int_0^t\uW'(s)\,ds\ge \uW(0)+\liminf_{t\to\infty}(t\tilde\beta)=\infty
\ee
which gives us a contradiction. Thus $T^*<\infty$, and in fact since, by (I'),
\be
\bal
\max\{\sup_\Omega v_1(0),\sup_\Omega v_2(0),\ogg\}-\epsilon\underline\alpha^\lambda\ge\underline V(T^*)-\epsilon\underline\alpha^\lambda\ge \uW(T^*)=&\uW(0)+\int_0^{T^*}\uW'(s)\,ds\\
\ge&\uW(0)+T^*\tilde\beta,
\eal
\ee
we have that $T^*$ is bounded above by a constant depending ultimately on $\delta, d_i, \sup_t\|b_i(t)\|_{L^\infty}$, the intitial and boundary conditions, and the domain,
\be
T^*\le \fr{1}{\tilde\beta}(\max\{\sup_\Omega v_1(0),\sup_\Omega v_2(0),\ogg\}-\uW(0)-\epsilon\underline\alpha^\lambda).
\ee
This completes the proof of the lower bound of (II') and thus of the proposition.
\end{proof}

We now prove Theorem \ref{maxthm} using Proposition \ref{prop!}.
\begin{proof}
We take $v_i=c_i$, $d_i=D_i$, $g_i=\gamma_i$, $p_i=c_i/\epsilon$, and
\be
b_i=-u+D_iz_i\na\Phi
\ee
in the proposition, and thus (I) follows. In order to show (II), it suffices to verify that $b_i$ is uniformly bounded in time. By (I) and (\ref{tpois}), we have that $\sup_t\|\na\Phi(t)\|_{L^\infty}<\infty.$ Thus it only remains to establish a uniform bound on $\|u\|_{L^\infty}$. To this end, we prove below that $\|Au\|_{L^2}$ is uniformly bounded in time, from which the desired result follows due to the embedding $H^2\hookrightarrow L^\infty.$

\textbf{Step 1. Uniform $L^\infty_tH^1_x$ bound on $u$.} Applying the Leray projection to (\ref{tstokes}), we obtain
\be
\pa_t u+\nu Au=-K\mathbb{P}(\rho\na\Phi).\la{lstokes}
\ee
Multiplying (\ref{lstokes}) by $Au$ and integrating by parts, we obtain, using (I),
\be
\bal
\fr{1}{2}\fr{d}{dt}\|u\|_V^2+\fr{\nu}{2}\|Au\|_H^2\le C'\|\mathbb{P}(\rho\na\Phi)\|_H^2\le C,
\eal
\ee
Then using the Stokes regularity estimate
\be
\|u\|_V\le C\|Au\|_H
\ee
we have
\be
\fr{d}{dt}\|u|_V^2\le -C\|u|_V^2+C
\ee
from which it follows that 
\be
\sup_t\|u(t)\|_V<\infty.\la{uv}
\ee

\textbf{Step 2. Local uniform $L^2_tH^1_x$ bounds on $c_i$.} Multiplying (\ref{tnp1}) by $c_1$ and integrating by parts we obtain
\be
\bal
\fr{1}{2}\fr{d}{dt}\|c_1\|_{L^2}^2-D_1\int_\Omega c_1\D c_1\,dV=D_1\int_\Omega c_1\na c_1\cdot\na\Phi-\fr{c_1^2\rho}{\epsilon}\,dV
\eal
\ee
and writing $\D c_1=\D(c_1-\Gamma_1)$ where $\Gamma_1$ is the unique harmonic function on $\Omega$ satisfying $\Gamma_1=\gamma_1$, we obtain after integrations by parts, Young's inequalities, and the uniform bounds on $c_i$,
\be
\fr{1}{2}\fr{d}{dt}\|c_1\|_{L^2}^2+\fr{D_1}{2}\|\na c_1\|_{L^2}^2\le C.
\ee
Then, integrating in time and again using the uniform bound on $c_1$, we obtain for all $t\ge 0$ and $\tau>0$, 
\be
\int_t^{t+\tau}\|\na c_1(s)\|_{L^2}^2\,ds\le C(1+\tau)
\ee
where $C$ is independent of $t$ and $\tau$. Similar estimates for $i=2$ give us
\be
\int_t^{t+\tau}\|\na c_i(s)\|_{L^2}^2\,ds\le \bar C(1+\tau),\quad i=1,2\la{lul2}
\ee
with $\bar C$ independent of $t$ and $\tau$.

\textbf{Step 3. Uniform $L^\infty_t H^1_x$ bounds on $c_i$} Multiplying (\ref{tnp1}) by $-\D c_1$, integrating by parts, and using uniform bounds on $c_i$, we obtain
\be
\fr{d}{dt}\|\na c_1\|_{L^2}^2+\|\D c_1\|_{L^2}^2\le C+C\|\na c_1\|_{L^2}^2.\la{db}
\ee
Now fix any $t>1$. By (\ref{lul2}), there exists $t_0\in(\lfloor t\rfloor-1,\lfloor t\rfloor)$ such that $\|\na c_1(t_0)\|_{L^2}^2\le 2\bar C$ (here $\lfloor t\rfloor$ denotes the largest integer not exceeding $t$). Then from (\ref{db}) and (\ref{lul2}), we have
\be
\bal
\|\na c_1(t)\|_{L^2}^2\le&\|\na c_1(t_0)\|_{L^2}^2+C(t-t_0)+C\int_{t_0}^t\|\na c_1(s)\|_{L^2}^2\,ds\\
\le& 2\bar C +2C+3C\bar C
\eal
\ee
where the final term does not depend on $t$. After similar estimates for $\na c_2$, we obtain
\be
\sup_t\|\na c_i(t)\|_{L^2}<\infty,\quad i=1,2.\la{h1b}
\ee

\textbf{Step 4. Local uniform $L^2_tH^2_x$ bounds on $c_i$}
Integrating (\ref{db}) and using (\ref{h1b}), we obtain
\be
\int_t^{t+\tau}\|\D c_1(s)\|_{L^2}^2\,ds\le C(1+\tau)
\ee
for $C$ independent of $t$ and $\tau$. The same method yields the corresponding estimate for $i=2$, and thus we have for $ C$ independent of $t$ and $\tau$ 
\be
\int_t^{t+\tau}\|\D c_i(s)\|_{L^2}^2\,ds\le  C(1+\tau),\quad i=1,2.\la{lul3}
\ee

\textbf{Step 5. Local uniform $L^2_tL^2_x$ bounds on $\pa_t c_i$.} Multiplying (\ref{tnp1}) by $\pa_t c_1$ and integrating by parts, we obtain, using the uniform bounds on $u$, $c_i$ and $\na c_i$,
\be\bal
\fr{D_1}{2}\fr{d}{dt}\|\na c_1\|_{L^2}^2+\fr{1}{2}\|\pa_t c_1\|_{L^2}^2&\le C(\|u\|_{V}^2\|\na c_1\|_{L^3}^2+\|\na c_1\|_{L^2}^2\|\na\Phi\|_{L^\infty}^2+\|c_1\rho\|_{L^2}^2)\\
&\le C(1+\|\D c_1\|_{L^2}^2)
\eal\ee
and integrating in time and using (\ref{h1b}), (\ref{lul3}), we obtain
\be
\int_t^{t+\tau}\|\pa_s c_1(s)\|_{L^2}^2\,ds\le C(1+\tau).
\ee
Similar estimates for $i=2$ give us
\be
\int_t^{t+\tau}\|\pa_s c_i(s)\|_{L^2}^2\,ds\le \tilde C(1+\tau),\quad i=1,2\la{ah}
\ee
for $\tilde C$ independent $t$ and $\tau$.

\textbf{Step 6. Local uniform $L^2_tL^2_x$ bounds on $\pa_t u$.} Multiplying (\ref{lstokes}) by $\pa_t u$ and integrating by parts, we have
\be
\fr{\nu}{2}\fr{d}{dt}\|u\|_V^2+\fr{1}{2}\|\pa_t u\|_H^2\le \|\rho\na\Phi\|_{L^2}^2\le C
\ee
and thus integrating in time, it follows from (\ref{uv}) that
\be
\int_{t}^{t+\tau}\|\pa_s u(s)\|_{H}^2\,ds\le C'(1+\tau)\la{ahh}
\ee
where $C'$ is independent of $t$ and $\tau$.

\textbf{Step 7. Uniform $L^\infty_tL^2_x$ bounds on $\pa_t c_i$ and $\pa_t u$.} Differentiating (\ref{tnp1}) in time, multiplying by $\pa_t c_1$ and integrating by parts, we obtain
\be
\bal
\fr{1}{2}\fr{d}{dt}\|\pa_t c_1\|_{L^2}^2+D_1\|\na \pa_t c_1\|_{L^2}^2=&-\int_\Omega (u\cdot\na\pa_t c_1)\pa_t c_1\,dV-\int_\Omega(\pa_t u\cdot\na c_1)\pa_t c_1\,dV\\
&-D_1\int_\Omega \pa_t c_1\na\Phi\cdot\na\pa_t c_1\,dV-D_1\int_\Omega c_1\na\pa_t\Phi\cdot\na\pa_t c_1\,dV. 
\eal
\ee
The first integral on the right hand side vanishes because $\div u=0$. We integrate the second integral by parts once more, and using Young's inequalities and uniform bounds on $c_i$ we obtain
\be
\fr{1}{2}\fr{d}{dt}\|\pa_t c_1\|_{L^2}^2+\fr{D_1}{2}\|\na\pa_t c_1\|_{L^2}^2\le C(\|\pa_t u\|_H^2+\|\pa_t c_1\|_{L^2}^2+\|\pa_t c_2\|_{L^2}^2).\la{11}
\ee
Similarly for $i=2$ we obtain
\be
\fr{1}{2}\fr{d}{dt}\|\pa_t c_2\|_{L^2}^2+\fr{D_2}{2}\|\na\pa_t c_2\|_{L^2}^2\le C(\|\pa_t u\|_H^2+\|\pa_t c_1\|_{L^2}^2+\|\pa_t c_2\|_{L^2}^2).\la{22}
\ee
Next differentiating (\ref{lstokes}) by time, multiplying by $\pa_t u$ and integrating by parts, we obtain
\be
\fr{1}{2}\fr{d}{dt}\|\pa_t u\|_H^2+\fr{\nu}{2}\|\pa_t u\|_V^2\le C(\|\pa_t\rho\|_{L^2}^2\|\na\Phi\|_{L^\infty}^2+\|\rho\|_{L^\infty}^2\|\na\pa_t\Phi\|_{L^2}^2)\le C(\|\pa_t c_1\|_{L^2}^2+\|\pa_t c_2\|_{L^2}^2).\la{33}
\ee
Now adding (\ref{11})-(\ref{33}), we obtain
\be
\fr{d}{dt}(\|\pa_t c_1\|_{L^2}^2+\|\pa_t c_2\|_{L^2}^2+\|\pa_t u\|_H^2)\le C(\|\pa_t c_1\|_{L^2}^2+\|\pa_t c_2\|_{L^2}^2+\|\pa_t u\|_H^2). \la{ahhh}
\ee
Finally, using the same method as in Step 3, we use (\ref{ah}), (\ref{ahh}) and (\ref{ahhh}) to obtain
\be
\sup_t(\|\pa_t c_1(t)\|_{L^2}+\|\pa_t c_2(t)\|_{L^2}+\|\pa_t u(t)\|_H)<\infty.
\ee

\textbf{Step 8. Uniform $L^\infty_t H^2_x$ bounds on $u$}. From (\ref{lstokes}), we have
\be
\|Au\|_H\le C(\|\pa_t u\|_{H}+\|\rho\na\Phi\|_{L^2})
\ee
and it follows from the preceding estimates that
\be
\sup_t\|Au(t)\|_H<\infty.
\ee
With this bound, the proof of the uniform boundededness of $b_i=-u+D_iz_i\na\Phi$ is complete, and thus (II) of the theorem follows from (II') of Proposition \ref{maxthm}.
\end{proof}

\section{Global Stability of Weak Steady Currents}\la{GS}
In this section we consider the long time behavior of solutions to the time dependent Nernst-Planck-Stokes (NPS) system  
\ref{tnp1})-(\ref{tdiv}). In this section  we take the domain to be the three dimensional periodic strip $\Omega=(0,L)\times\mathbb{T}\times\mathbb{T}$ where $\mathbb{T}$ has period $1$, and boundary conditions
\begin{align}
    c_i(t,0,y,z)&=\alpha_i,\quad c_i(t,L,y,z)=\beta_i\la{bc1}\\
    \Phi(t,0,y,z)&=-V,\quad \Phi(t,L,y,z)=0\la{bc2}\\
    u(t,0,y,z)&=u(t,L,y,z)=0.\la{bc3}
\end{align}
Here, we take $\alpha_i,\beta_i,V>0$ to be \textit{constants}.

In the first subsection of this section, we analyze one dimensional solutions to NPS with boundary conditions (\ref{bc1})-(\ref{bc3}) and establish uniform bounds.  In the second subsection, we show that \textit{weak current} one dimensional solutions are globally stable. This latter result yields as a corollary the uniqueness of steady state solutions in the setting of small perturbations from equilibrium.

\subsection{One Dimensional Steady States}\la{odss}
We consider the one dimensional steady state Nernst-Planck system for $x\in(0,L)$
\begin{align}
    0&=\pa_x(\pa_x c_1^*+c_1^*\pa_x\Phi^*)\la{1np1}\\
    0&=\pa_x(\pa_x c_2^*-c_2^*\pa_x\Phi^*)\la{1np2}\\
    -\epsilon\pa_{xx}\Phi^*&=\rho^*=c_1^*-c_2^*\la{1pois}
\end{align}
with boundary conditions corresponding to (\ref{bc1}), (\ref{bc2})
\begin{align}
    c_i^*(0)&=\alpha_i>0,\quad c_i^*(L)=\beta_i>0\la{BC1}\\
    \Phi^*(0)&=-V<0,\quad \Phi^*(L)=0.\la{BC2}
\end{align}
As we will see in Section \ref{gs}, one dimensional Nernst-Planck steady states, with zero fluid flow $u^*\equiv 0$, are also steady state solutions to the full three dimensional NPNS system in our current setting of a three dimensional periodic strip with boundary conditions (\ref{bc1})-(\ref{bc3}). This is the motivation for the study of these one dimensional solutions.

While the computations of the previous section could be significantly simplified for this one dimensional, no fluid setting, nonetheless the existence of a smooth solution to (\ref{1np1})-(\ref{BC2}), with $c_i^*\ge 0$, follows from a streamlined version of the proof of Theorem \ref{thm1} (see also \cite{mock}). 

In this subsection, we establish uniform bounds on $c_i^*$ and $\Phi^*$ that depend exclusively on boundary data. To this end, we recall the electrochemical potentials
\be
\mu_i^*=\log c_i^*+z_i\Phi^*\la{mu}
\ee
and the related variables (a.k.a. Slotboom variables in the semiconductor literature)
\be
\eta_i^*=\exp\mu_i^*=c_i^*e^{z_i\Phi^*}\la{eta}
\ee
where $z_1=1=-z_2$. We refer the reader to \cite{park} for a more complete study on the one dimensional steady state Nernst-Planck system.

\begin{prop}
Suppose $(c_1^*,c_2^*,\Phi^*)$ is a smooth solution to (\ref{1np1})-(\ref{BC2}). Then, the solution satisfies the following uniform bounds:
\begin{enumerate}[(I)]
    \item $\min\{\alpha_ie^{-z_iV},\beta_i\}=\lambda_i\le \eta_i^*\le\Lambda_i=\max\{\alpha_ie^{-z_iV},\beta_i\}$
    \item $\min\{-V,\log(\lambda_1/\Lambda_2)^\fr{1}{2}\}=-v\le\Phi^*\le\mathcal{V}= \max\{0,\log(\Lambda_1/\lambda_2)^\fr{1}{2}\}$
    \item $\min\{\alpha_1,\alpha_2,\beta_1,\beta_2\}=\underline{\gamma}\le c_i^*\le\og=\max\{\alpha_1,\alpha_2,\beta_1,\beta_2\} \,\text{ for }\, i =1,2.$
\end{enumerate}\la{prop}
\end{prop}

\begin{proof}
Writing (\ref{1np1}), (\ref{1np2}) in terms of $\eta_i^*$ we have
\be
0=\pa_x(e^{-z_i\Phi^*}\pa_x\eta_i^*)=e^{-z_i\Phi^*}(-z_i\pa_x\Phi^*\pa_x\eta_i^*+\pa_{xx}\eta_i^*).
\ee
And thus (I) follows from the weak maximum principle and the fact that
\be
\eta_i^*(0)=\alpha_ie^{-z_iV},\quad \eta_i^*(L)=\beta_i.
\ee
To show (II), we rewrite (\ref{1pois}) as 
\be
-\epsilon\pa_{xx}\Phi^*=\eta_1^*e^{-\Phi^*}-\eta_2^*e^{\Phi^*}.
\ee
So if $\Phi^*$ attains its global maximum at an interior point $x_0\in(0,L)$, then
\be
\bal
&0\le -\epsilon\pa_{xx}\Phi^*(x_0)=\eta_1^*(x_0)e^{-\Phi^*(x_0)}-\eta_2^*(x_0)e^{\Phi^*(x_0)}\\
\Rightarrow&\Phi^*(x_0)\le\log\left(\fr{\eta_1^*(x_0)}{\eta_2^*(x_0)}\right)^\fr{1}{2}\le \log(\Lambda_1/\lambda_2)^\fr{1}{2}
\eal
\ee
and the upper bound in (II) follows. Similarly, if $\Phi^*$ attains its global minimum at an interior point $x_0\in(0,L)$, then
\be
\bal
&0\ge\eta_1^*(x_0)e^{-\Phi^*(x_0)}-\eta_2^*(x_0)e^{\Phi^*(x_0)}\\
\Rightarrow&-\Phi^*(x_0)\le \log\left(\fr{\eta_2^*(x_0)}{\eta_1^*(x_0)}\right)^\fr{1}{2}\le\log (\Lambda_2/\lambda_1)^\fr{1}{2}
\eal
\ee
and the lower bound in (II) follows. 

Lastly, prior to proving (III), we note that by combining (I) and (II) and using the definition of $\eta_i^*$, it is possible to obtain upper and lower bounds on $c_i^*$ that depend on boundary data for $c_i$ and $\Phi^*$. Here, instead we establish the bounds in (III), which in particular does not depend on boundary data for $\Phi^*$. We prove only the upper bound as the lower bound can be shown analogously. To do so, we introduce the rescaling $X=x/\epsilon^\fr{1}{2}$ so that we can rewrite (\ref{1np1}), (\ref{1np2}) as
\begin{align}
-\pa_{XX}c_1^*=&\pa_Xc_1^*\pa_X\Phi^*-c_1^*(c_1^*-c_2^*)\la{X1}\\
-\pa_{XX}c_2^*=&-\pa_Xc_2^*\pa_X\Phi^*+c_2^*(c_1^*-c_2^*).\la{X2}
\end{align}
Suppose that $\max\{c_1^*,c_2^*\}$ attains a global maximal value, $c>\og$, at an interior point $X_0\in (0,L/\epsilon^\fr{1}{2})$. Assume without loss of generality that this maximum is attained by $c_1^*$. Then we have
\be
\bal
0\le-\pa_{XX}c_1^*(X_0)=&\pa_Xc_1^*(X_0)\pa_X\Phi^*(x_0)-c_1^*(X_0)(c_1^*(X_0)-c_2^*(X_0))\\
=&-c(c-c_2^*(X_0)).\la{118}
\eal
\ee
Then since by assumption we have $c\ge c_2^*(X_0)$, we necessarily have that $c=c_2^*(X_0)$, for otherwise, the right hand side of (\ref{118}) becomes strictly negative. Furthermore, the inequality in (\ref{118}) is an equality, and we conclude that $\pa_{XX}c_1^*(X_0)=0$. And since we have shown that $c_2^*$ also attains its global maximum at $X_0$, by evaluating (\ref{X2}) at $X_0$ we conclude that $\pa_Xc_2^*(X_0)=\pa_{XX}c_2^*(X_0)=0$.

It follows by induction that for $i = 1,2$ we have 
\be
\pa_X^{k}c_i^*(X_0)=0 \quad\text{for all } k\ge 1.\la{d}
\ee
Indeed, assume, for the sake of induction, that this is true for all $1\le k'\le k$ where $k\ge 2$. Then differentiating (\ref{X1}),
\be
-\pa_{X}^{k+1}c_1^*=\sum_{j=0}^{k-1}\binom{k-1}{j}\left(\pa_X^{j+1}c_1^*\pa_X^{k-j}\Phi^*-\pa_X^jc_1^*\pa_X^{k-1-j}(c_1^*-c_2^*)\right)  \la{k+1}
\ee
and evaluating (\ref{k+1}) at $X_0$ and using the induction hypothesis together with the fact that $c_1^*(X_0)=c_2^*(X_0)$, we conclude $\pa_X^{k+1}c_1^*(X_0)=0$ as desired. Similarly by differentiating (\ref{X2}) we obtain $\pa_X^{k+1}c_2^*(X_0)=0$. 

Now \textit{if} $c_i^*$ is real analytic, then (\ref{d}) implies that in fact $c_1^*\equiv c_2^*\equiv c$, but we assumed that $c>\og$, so this contradiction implies the upper bound $\max\{c_1^*,c_2^*\}\le\og$.

So to complete the proof of (III) it suffices to establish the real analyticity of $c_i^*$. To this end, we prove that there exists $C>0$ such that for all integers $k\ge 0$
\be
\bal
A^{-1}&=\sup_X|\pa_X\Phi^*|\le \fr{1}{4}C\\
A^k&=\sup_X|\pa_X^kc_1^*|+\sup_X|\pa_X^kc_2^*|\le \fr{1}{4}(k+1)!C^{k+2}.\la{Ak}
\eal
\ee
We choose $C$ so that (\ref{Ak}) holds for $A^k$, $k=-1,0,1$. Now we prove the implication
\be
(\ref{Ak}) \text{ holds for all $k'$ less than or equal to $k\ge 1$}\Rightarrow (\ref{Ak}) \text{ holds for $k+1$}. 
\ee
To show this, we see from (\ref{k+1}) that
\be
\bal
|\pa_X^{k+1}c_i^*|\le&\sum_{j=0}^{k-1}\binom{k-1}{j}(A^{j+1}A^{k-j-2}+A^{j}A^{k-1-j})\\
\le&\fr{1}{16}\sum_{j=0}^{k-1}\fr{(k-1)!}{j!(k-1-j)!}((j+2)!(k-j-1)!C^{k+3}+(j+1)!(k-j)!C^{k+3})\\
=&\fr{C^{k+3}}{16}(k-1)!\sum_{j=0}^{k-1}((j+2)(j+1)+(j+1)(k-j))\\
\le&\fr{C^{k+3}}{16}(k-1)!k((k+1)k+k^2)\\
\le&\fr{1}{8}(k+2)!C^{k+3}
\eal
\ee
and summing in $i$ we obtain
\be
A^{k+1}\le \fr{1}{4}(k+2)!C^{k+3}
\ee
as desired. Thus $c_i^*$ is real analytic and the proof of the upper bound in (III) and of the proposition is complete.
\end{proof}
\begin{rem}
Proposition \ref{prop} (I), (II)  directly extend to higher dimensional settings when no fluid is involved (see also \cite{mock} for further generalizations). On the other hand, it is unclear if (III) holds in higher dimensions as the proof provided above relies on a property of one dimensional real analytic functions. 
\end{rem}

\subsection{Global Stability}\la{gs}
In this last subsection, we study the problem of stability of one dimensional steady currents on the domain $\Omega=(0,L)\times\mathbb{T}^2$.

\begin{defi}
We say that $(c_1^*, c_2^*, u^*\equiv 0)$ is a \textbf{(one dimensional) steady current} solution to (\ref{tnp1})-(\ref{tdiv}) with boundary conditions (\ref{bc1})-(\ref{bc3}) on the three dimensional periodic strip $\Omega=(0,L)\times \mathbb{T}^2$ if $c_i^*$ is independent of the spatial variables $y$ and $z$, independent of time, and solves the one dimensional problem (\ref{1np1})-(\ref{BC2}).
\end{defi}

\begin{rem} A solution $c_i^*(x)$ to the one dimensional system (\ref{1np1})-(\ref{BC2}), seen as a three dimensional function on $\Omega$, independent of $y, z$, together with $u^*\equiv 0$, is indeed a solution to the three dimensional steady state NPNS system because
\be
\bal
\pa_xc_1^*+c_1^*\pa_x\Phi^*&=j_1\\
\pa_xc_2^*-c_2^*\pa_x\Phi^*&=j_2\la{j}
\eal
\ee
for constants $j_i$. It follows from this that
\be
\rho^*\na\Phi^*=-\na(c_1^*+c_2^*)+(j_1+j_2,0,0)=-\na(c_1^*+c_2^*-(j_1+j_2)x),
\ee
and thus $u^*(x,y,z)\equiv 0$ and $p^*(x,y,z)=K(c_1^*(x)+c_2^*(x)-(j_1+j_2)x)$ solve the Stokes equations (\ref{tstokes})-(\ref{tdiv}).
\end{rem}

In general, it is not known whether solutions to the one dimensional system (\ref{1np1})-(\ref{BC2}) are unique, and therefore it is also unknown whether any given one dimensional steady current solution $(c_1^*, c_2^*,u^*\equiv 0)$ to the three dimensional NPNS system is unique (in general one cannot rule out the existence of other one dimensional steady current solutions nor the existence of solutions that depend also on $y$ and/or $z$).

For the remainder of this subsection, we study the stability of a \textit{fixed} one dimensional steady current solution. However, using the a priori estimates of Section \ref{odss} and under a \textit{weak current} or \textit{small perturbation from equilibrium} assumption, c.f. (\ref{jsmall}), we obtain the global stability of the fixed steady current solution (Theorem \ref{globalstab}). As a consequence of stability it follows that the fixed steady current solution is  the unique steady state solution of the full three dimensional system (\ref{tnp1})-(\ref{tdiv}) with boundary conditions (\ref{bc1})-(\ref{bc3}) (Theorem \ref{unique}). 

\begin{rem}
It seems somewhat unusual to first establish the global stability of a steady state solution, and then its uniqueness. This is due to the absence of certain a priori information. This difficulty does not arise if we consider the Nernst-Planck system, uncoupled to Navier-Stokes equations. In this case, a straightforward generalization of the estimates in Section \ref{odss} give simple, explicit a priori bounds on steady state solutions depending also on $y,z$,  thus allowing uniqueness to be established independently of global stability.\la{lastrem}
\end{rem}

The main tool in proving global stability is the following log-Sobolev type inequality, which is also used in \cite{gajstab} in an equilibrium setting.
\begin{lem}
Suppose $f_i,g_i$, $i=1,2$ and $p^f, p^g$ are smooth real valued functions on a bounded domain $\Omega\subset\mathbb{R}^3$ satisfying the bounds
\be
0<f_1, f_2\le M_f,\quad 0< g_1, g_2\le M_g. \la{ab}
\ee
and the relations
\be
\bal
{f_1}_{|\pa\Omega}&={g_1}_{|\pa\Omega}\\
{f_2}_{|\pa\Omega}&={g_2}_{|\pa\Omega}\\
{p^f}_{|\pa\Omega}&={p^g}_{|\pa\Omega}\\
-\epsilon\D p^f&= f_1-f_2\\
-\epsilon\D p^g&= g_1-g_2,\quad i=1,2.\la{rel}
\eal
\ee
Then the functions
\be
\pi^f_i=\log f_i+ z_ip^f,\quad \pi_i^g=\log g_i+z_ip^g,\quad i=1,2
\ee
where $z_1=1=-z_2$, satisfy the bound
\be
\fr{\omega}{l^2}\left(\sum_{i=1}^2\fr{1}{2}\int_\Omega g_i\psi\left(\fr{f_i}{g_i}\right)\,dV+\epsilon\int_\Omega |\na(p^f-p^g)|^2\,dV\right)\le \sum_{i=1}^2\int_\Omega |\na(\pi_i^f-\pi_i^g)|^2\,dV
\ee
where
\be
\psi(s)=s\log s-s+1, \quad s>0\la{psi}
\ee
\be
\omega=\fr{2}{\max\{M_f, M_g\}}\la{omega}
\ee
and $l$ can be chosen to be the height of any infinite slab in $\mathbb{R}^3$ that contains $\Omega$ (i.e. $\Omega\subset\{x_0+s_1e_1+s_2e_2+s_3e_3\,|\, s_1\in (0,l),\, s_2,s_3\in(-\infty,\infty)\}$ for some $x_0\in\mathbb{R}^3$ and orthonormal basis $\{e_i\}$ of $\mathbb{R}^3).$\la{poin}
\end{lem}

Prior to proving Lemma \ref{poin}, we first establish an interpolation inequality that interpolates $L^2$ between $L\log L$ and $L^\infty$.

\begin{lem}
For positive, real valued, bounded, measurable functions $f, g$  defined on $\Omega$, we have
\be
\int_\Omega (f-g)^2\,dV\le \max\{\|f\|_{L^\infty},\|g\|_{L^\infty}\}\int_\Omega g\psi\left(\fr{f}{g}\right)\,dV
\ee
with $\psi$ defined in (\ref{psi}).\la{inter}
\end{lem}
\begin{proof}
Taylor expanding $\psi(s)$ around $s=1$, we have
\be
\psi(s)\ge \min\{1,s^{-1}\}(s-1)^2\Rightarrow (s-1)^2\le \max\{1,s\}\psi(s),
\ee
so taking $s=f/g$ we have
\be
\bal
&\left(\fr{f}{g}-1\right)^2\le\max\left\{1,\fr{f}{g}\right\}\psi\left(\fr{f}{g}\right)\\
\Rightarrow&(f-g)^2\le\max\{f,g\}g\psi\left(\fr{f}{g}\right)
\eal
\ee
and thus the lemma follows after integrating over $\Omega$.
\end{proof}

Now we prove Lemma \ref{poin}.
\begin{proof}
We consider the following expression
\be
\sum_{i=1}^2\left<f_i-g_i,\pi^f_i-\pi^g_i\right>.
\ee
On one hand we have, using the Poisson equation (\ref{rel}),
\be
\begin{aligned}
    \sum_i\left<f_i-g_i,\pi^f_i-\pi^g_i\right>=&\sum_i\left<f_i-g_i,\log\fr{f_i}{g_i}+z_i(p^f-p^g)\right>\\
    =&\sum_i\int_\Omega g_i\left(\fr{f_i}{g_i}-1\right)\log\fr{f_i}{g_i}\,dV+\epsilon\int_\Omega|\na(p^f-p^g)|^2\,dV\\
    \ge& \sum_i\int_\Omega g_i\psi\left(\fr{f_i}{g_i}\right)\,dV+\epsilon\int_\Omega|\na(p^f-p^g)|^2\,dV\la{lb}
\end{aligned}
\ee
where in the last inequality we used the inequality
\be
(s-1)\log s\ge \psi(s),\quad s>0.
\ee
On the other hand, we have due to Young's inequality, Poincaré's inequality, and Lemma \ref{inter},
\be
\begin{aligned}
\sum_i\left<f_i-g_i,\pi^f_i-\pi^g_i\right>\le & \sum_i \fr{\omega}{4}\|f_i-g_i\|_{L^2}^2+\sum_i\fr{1}{\omega}\|\pi^f_i-\pi^g_i\|_{L^2}^2\\
\le&\sum_i\fr{\omega\max\{M_f,M_g\}}{4}\int_\Omega g_i\psi\left(\fr{f_i}{g_i}\right)\,dV+\sum_i\fr{l^2}{\omega}\|\na(\pi^f_i-\pi^g_i)\|_{L^2}^2.\la{ub}
\end{aligned}
\ee
Therefore, choosing $\omega>0$ as in (\ref{omega}), we have, combining (\ref{lb}) and (\ref{ub}),
\be
\fr{\omega}{l^2}\left(\sum_i\fr{1}{2}\int_\Omega g_i\psi\left(\fr{f_i}{g_i}\right)\,dV+\epsilon\int_\Omega |\na(p^f-p^g)|^2\,dV\right)\le \sum_i\int_\Omega|\na(\pi^f_i-\pi^g_i)|^2\,dV.
\ee
\end{proof}

We now state the global stability theorem.

\begin{thm}
Suppose $(c_1^*, c_2^*, u^*\equiv 0)$ is a one dimensional steady current solution to the NPS system (\ref{tnp1})-(\ref{tdiv}) with boundary conditions (\ref{bc1})-(\ref{bc3}). Suppose furthermore that the corresponding \textbf{p-current} $j_1$ and \textbf{n-current} $j_2$ (c.f. (\ref{j})) satisfy
\begin{align}
\max_i |j_i|LG_i<\fr{1}{\sqrt{2}} \la{jsmall}
\end{align}
where 
\be
G_i=\sqrt{\fr{1}{D}\left(\fr{D_i\og^2}{2\ug^4}+\fr{KL^2\og^2}{\nu\ug^3}\right)}
\ee
and $D=\min_i D_i$. Then $(c_1^*, c_2^*, u^*\equiv 0)$ is globally asymptotically stable. That is, for any smooth initial conditions $c_1(0)\ge 0, c_2(0)\ge 0, u(0)$ (with $\div u(0)=0$), the corresponding solution $(c_1, c_2, u)$ to (\ref{tnp1})-(\ref{tdiv}) satisfies
\be
\sum_{i=1}^2\int_\Omega c_i^*\psi\left(\fr{c_i(t)}{c_i^*}\right)\,dV+\int_\Omega |\na(\Phi-\Phi^*)|^2\,dV+\int_\Omega |u(t)|^2\,dV\to 0\quad \text{as } t\to\infty.\la{conv}
\ee
Furthermore, there exists $T^*>0$, depending on initial and boundary data and the parameters of the system, such that after time $t=T^*$, the rate of convergence in (\ref{conv}) is exponential in time.\la{globalstab}
\end{thm}
A consequence of Theorem \ref{globalstab} is the following uniqueness theorem.
\begin{thm}
Under the same hypotheses as in Theorem \ref{globalstab}, the one dimensional steady current solution $(c_1^*, c_2^*, u^*\equiv 0)$ is the unique steady state solution to the NPS system (\ref{tnp1})-(\ref{tdiv}) with boundary conditions (\ref{bc1})-(\ref{bc3}).\la{unique}
\end{thm}
\begin{rem}
We note that the currents $j_i$ are solution dependent constants and the condition \eqref{jsmall} is not explicitly written  solely in terms of the boundary data. Writing (\ref{j}) in terms of the electrochemical potentials and the Slotboom variables (c.f. (\ref{mu}), (\ref{eta})) we have
\be
c_i^*\pa_x\mu_i^*=j_i,\quad e^{-z_i\Phi^*}\pa_x\eta_i^*=j_i\la{jj}
\ee
and thus
\be
j_i=\fr{\mu_i^*(L)-\mu_i^*(0)}{\int_0^L\fr{1}{c_i^*}\,dx}=\fr{\eta_i^*(L)-\eta_i^*(0)}{\int_0^L e^{z_i\Phi^*}\,dx}.
\ee
Then, using the uniform bounds from Proposition \ref{prop}, we see that explicit sufficient conditions in terms of the boundary data 
which imply the smallness conditions (\ref{jsmall}) are given by 
\be\begin{cases}
 \left|\log\fr{\alpha_1}{\beta_1}+V\right|\og G_1<\fr{1}{\sqrt 2}\\
\left|\log\fr{\alpha_2}{\beta_2}-V\right|\og G_2<\fr{1}{\sqrt 2}
\end{cases}\ee
or
\be\begin{cases}
 \left|\alpha_1-\beta_1e^{-V}\right|e^{v}G_1<\fr{1}{\sqrt 2}\\
\left|\alpha_2-\beta_2e^V\right|e^{\mathcal{V}}G_2<\fr{1}{\sqrt 2}
\end{cases}\ee
where $v, \mathcal{V}$ are defined in Proposition \ref{prop}.
\end{rem}

Now we prove Theorem \ref{globalstab}.

\begin{proof}
First suppose that the initial conditions satisfy
\be
0<\ugd\le c_i(0)\le \ogd,\quad i=1,2
\ee
where
\be
\ugd=\ug-\delta,\quad\ogd=\og+\delta
\ee
for some small $\delta> 0$, to be determined below (c.f. (\ref{del2})). Then by Proposition \ref{prop} and Theorem \ref{maxthm} (I), the time dependent solution $(c_1, c_2, u)$ and the one dimensional steady current $(c_1^*, c_2^*, u^*\equiv 0)$ satisfy the bounds
\be
\bal
\ugd\le c_i(t)\le \ogd,\quad \ug\le c_i^*\le\og.\la{cbounds}
\eal
\ee
Next, writing (\ref{tnp1}), (\ref{tnp2}) in terms of the electrochemical potentials
\be
\mu_i=\log c_i+z_i\Phi,\quad i=1,2
\ee
and in terms of the differences $c_i-c_i^*$, $\mu_i-\mu_i^*$, we have
\be
\bal
\pa_t(c_1-c_1^*)&=-u\cdot\na c_1+D_1\div(c_1\na(\mu_1-\mu_1^*)+(c_1-c_1^*)\na\mu_1^*)\\
\pa_t(c_2-c_2^*)&=-u\cdot\na c_2+D_2\div(c_2\na(\mu_2-\mu_2^*)+(c_2-c_2^*)\na\mu_2^*).
\eal
\ee
We multiply the above equations by $\mu_1-\mu_1^*$ and $\mu_2-\mu_2^*$, respectively, and integrate by parts. On the left hand side, we obtain, after summing in $i$, 
\be
\bal
\sum_i\left<\pa_t(c_i-c_i^*),\mu_i-\mu_i^*\right>&=\sum_i\left(\fr{d}{dt}\int_\Omega c_i^*\psi\left(\fr{c_i}{c_i^*}\right)\,dV+z_i\left<\pa_t(c_i-c_i^*),\Phi-\Phi^*\right>\right)\\
&=\sum_i\fr{d}{dt}\int_\Omega c_i^*\psi\left(\fr{c_i}{c_i^*}\right)\,dV+\left<\pa_t(\rho-\rho^*),\Phi-\Phi^*\right>\\
&=\sum_i\fr{d}{dt}\int_\Omega c_i^*\psi\left(\fr{c_i}{c_i^*}\right)\,dV+\fr{\epsilon}{2}\fr{d}{dt}\|\na(\Phi-\Phi^*)\|_{L^2}^2.
\eal
\ee
On the right hand side, for $i=1$, we have, using Lemma \ref{inter}, (\ref{cbounds}), and (\ref{jj}),
\be
\begin{aligned}
&\left<-u\cdot\na c_1+D_1\div(c_1\na(\mu_1-\mu_1^*)+(c_1-c_1^*)\na\mu_1^*),\mu_1-\mu_1^*\right>\\
=&-\left<u\cdot\na  c_1,\mu_1-\mu_1^*\right>-D_1\int_\Omega c_1|\na(\mu_1-\mu_1^*)|^2\,dV-D_1\left<(c_1-c_1^*)\na\mu_1^*,\na(\mu_1-\mu_1^*)\right>\\
\le& -\fr{D_1}{2}\int_\Omega c_1|\na(\mu_1-\mu_1^*)|^2\,dV-\left<u\cdot\na  c_1,\mu_1-\mu_1^*\right>+\fr{D_1}{2}\int_\Omega \fr{(c_1-c_1^*)^2}{c_1}(\pa_x\mu_1^*)^2\,dV\\
\le& -\fr{D_1}{2}\int_\Omega c_1|\na(\mu_1-\mu_1^*)|^2\,dV-\left<u\cdot\na  c_1,\mu_1-\mu_1^*\right>+\fr{D_1\ogd}{2\ugd\ug^2}j_1^2\int_\Omega c_1^*\psi\left(\fr{c_1}{c_1^*}\right)\,dV.\la{calc1}
\end{aligned}
\ee
We take a closer look at the term involving $u$, integrating by parts and using $\div u=0$,
\be
\bal
-\left<u\cdot\na c_1,\mu_1-\mu_1^*\right>=&-\left<u\cdot\na c_1,\log c_1+\Phi\right>+\left<u\cdot\na c_1,\mu_1^*\right>\\
=&-\left<u,\na(c_1\log c_1-c_1)\right>+\int_\Omega uc_1\cdot\na\Phi\,dV\\
&+\left<u\cdot\na(c_1-c_1^*),\mu_1^*\right>+\left<u \cdot\na c_1^*,\mu_1^*\right>\\
=&\int_\Omega uc_1\cdot\na\Phi\,dV-\left<u(c_1-c_1^*),\na\mu_1^*\right>\\
&+\left<u,\na(c_1^*\log c_1^*-c_1^*)\right>-\int_\Omega uc_1^*\cdot\na\Phi^*\,dV\\
\le&\int_\Omega uc_1\cdot\na\Phi\,dV-\int_\Omega uc_1^*\cdot\na\Phi^*\,dV\\
&+\fr{\nu}{4KL^2}\|u\|_H^2+\fr{KL^2}{\nu}\int_\Omega|c_1^*\pa_x\mu_1^*|^2\fr{(c_1-c_1^*)^2}{(c_1^*)^2}\,dV\\
\le&\int_\Omega uc_1\cdot\na\Phi\,dV-\int_\Omega uc_1^*\cdot\na\Phi^*\,dV\\
&+\fr{\nu}{4K}\|u\|_V^2+\fr{KL^2\ogd}{\nu\ug^2}j_1^2\int_\Omega c_1^*\psi\left(\fr{c_1}{c_1^*}\right)\,dV
\eal
\ee
and thus returning to (\ref{calc1}), we have
\be
\begin{aligned}
&\left<-u\cdot\na c_1+D_1\div(c_1\na(\mu_1-\mu_1^*)+(c_1-c_1^*)\na\mu_1^*),\mu_1-\mu_1^*\right>\\
\le& -\fr{D_1}{2}\int_\Omega c_1|\na(\mu_1-\mu_1^*)|^2\,dV+\left(\fr{D_1\ogd}{2\ugd\ug^2}+\fr{KL^2\ogd}{\nu\ug^2}\right)j_1^2\int_\Omega c_1^*\psi\left(\fr{c_1}{c_1^*}\right)\,dV\\
&+\fr{\nu}{4K}\|u\|_V^2+\int_\Omega uc_1\cdot\na\Phi\,dV-\int_\Omega uc_1^*\cdot\na\Phi^*\,dV.
\eal
\ee
Similarly for $i=2$ we obtain
\be
\begin{aligned}
&\left<-u\cdot\na c_2+D_2\div(c_2\na(\mu_2-\mu_2^*)+(c_2-c_2^*)\na\mu_2^*),\mu_2-\mu_2^*\right>\\
\le& -\fr{D_2}{2}\int_\Omega c_2|\na(\mu_2-\mu_2^*)|^2\,dV+\left(\fr{D_2\ogd}{2\ugd\ug^2}+\fr{KL^2\ogd}{\nu\ug^2}\right)j_2^2\int_\Omega c_2^*\psi\left(\fr{c_2}{c_2^*}\right)\,dV\\
&+\fr{\nu}{4K}\|u\|_V^2-\int_\Omega uc_2\cdot\na\Phi\,dV+\int_\Omega uc_2^*\cdot\na\Phi^*\,dV.
\end{aligned}
\ee
Collecting our estimates thus far and using the fact that $\rho^*\na\Phi^*=-\na(c_1^*+c_2^*-(j_1+j_2)x)$ is a gradient, we have
\be
\begin{aligned}
&\fr{d}{dt}\mathcal{E}+\sum_i\fr{D_i}{2}\int_\Omega c_i|\na(\mu_i-\mu_i^*)|^2\,dV\\
\le&\sum_iM_i^\delta j_i^2\int_\Omega c_i^*\psi\left(\fr{c_i}{c_i^*}\right)\,dV+\int_\Omega u\rho\cdot\na\Phi\,dV-\int_\Omega u\rho^*\cdot\na\Phi^*\,dV+\fr{\nu}{2K} \|u\|_{V}^2\\
=&\sum_iM_i^\delta j_i^2\int_\Omega c_i^*\psi\left(\fr{c_i}{c_i^*}\right)\,dV+\int_\Omega u\rho\cdot\na\Phi\,dV+\fr{\nu}{2K} \|u\|_{V}^2\la{E1}
\end{aligned}
\ee
where (see \cite{ci})
\be
\bal
\mathcal{E}&=\sum_i\int_\Omega c_i^*\psi\left(\fr{c_i}{c_i^*}\right)\,dV+\fr{\epsilon}{2}\|\na(\Phi-\Phi^*)\|_{L^2}^2\\
M_i^\delta&=\fr{D_i\ogd}{2\ugd\ug^2}+\fr{KL^2\ogd}{\nu\ug^2}.
\eal
\ee

Now we take a look at the Stokes equations,
\be
\pa_t u+\nu A u=-K\mathbb{P}(\rho\na\Phi).
\ee
Multiplying by $\fr{u}{K}$ and integrating by parts, we obtain using the self-adjointness of $\mathbb{P}$,
\be
\begin{aligned}
\fr{1}{K}\fr{d}{dt}\|u\|_{H}^2+\fr{\nu}{K}\| u\|_{V}^2=& -\int_\Omega u\cdot\mathbb{P}(\rho\na\Phi)\,dV=-\int_\Omega u\rho\cdot\na\Phi\,dV.\la{u}
\end{aligned}
\ee
Now we combine the estimates (\ref{E1}) and (\ref{u}) to obtain
\be
\bal
\fr{d}{dt}\left(\mathcal{E}+\fr{1}{K}\|u\|_{H}^2\right)+\sum_i\fr{D_i}{2}\int_\Omega c_i|\na(\mu_i-\mu_i^*)|^2\,dV+ \fr{\nu}{2K}\| u\|_{V}^2
\le&\sum_iM_i^\delta j_i^2\int_\Omega c_i^*\psi\left(\fr{c_i}{c_i^*}\right)\,dV.\la{finalE}
\eal
\ee
Now applying Lemma \ref{poin} to the dissipation term
\be
\mathcal{D}=\sum_i\fr{D_i}{2}\int_\Omega c_i|\na(\mu_i-\mu_i^*)|^2\,dV
\ee
we obtain
\be
\bal
\mathcal{D}\ge&\fr{D\ugd}{\ogd L^2}\left(\sum_i\fr{1}{2}\int_\Omega c_i^*\psi\left(\fr{c_i}{c_i^*}\right)\,dV+\epsilon\|\na(\Phi-\Phi^*)\|_{L^2}^2\right)\ge\fr{D\ugd}{2\ogd L^2}\mathcal{E}
\eal
\ee
where $D=\min_i D_i$. Next, defining
\be
\kappa_i^\delta=\fr{D\ugd}{2\ogd L^2}-M_i^\delta j_i^2,\la{del2}
\ee
we have, due to (\ref{jsmall}), $\kappa_i^\delta >0$ for each $i$ for small enough $\delta>0$. Therefore, after an application of Poincaré's inequality to $\|u\|_V^2$ in (\ref{finalE}), we obtain, for small enough $\delta,$
\be
\fr{d}{dt}\mathcal{F}\le -\kappa^\delta \mathcal{F}
\ee
for
\be
\mathcal{F}=\mathcal{E}+\fr{1}{K}\|u\|_H^2\la{F}
\ee
and
\be
\kappa^\delta=\min\{\kappa_1^\delta,\kappa_2^\delta,\nu/(2L^2)\}>0.
\ee
It follows that 
\be\mathcal{F}(t)\le \mathcal{F}(0)e^{-\kappa^\delta t}.\la{exp}\ee

Now, for general initial conditions, it suffices to observe that due to Theorem \ref{maxthm} (II), there exists some time $T^*>0$ such that $\ugd\le c_1(T^*), c_2(T^*)\le \ogd$, and then the convergence result follows from the preceding analysis by taking $c_i(T^*)$ to be the initial conditions. This completes the proof of the theorem.
\end{proof}

Lastly we prove Theorem \ref{unique}.

\begin{proof}
Suppose $(\oc_1,\oc_2,\ou)\neq (c_1^*, c_2^*, 0)$ is a steady state solution of the NPS system (\ref{tnp1})-(\ref{tdiv}) with boundary conditions (\ref{bc1})-(\ref{bc3}). Then by taking initial conditions $c_i(0)=\oc_i,\,u(0)=\ou$, we find that the corresponding energy $\mathcal{F}(t)$, defined in (\ref{F}), is constant in time and positive, $\mathcal{F}(t)=\mathcal{F}(0)>0$. On the other hand, Theorem \ref{globalstab} implies that $\mathcal{F}(t)\to 0$ as $t\to\infty$, and thus we have a contradiction. This contradiction completes the proof.
\end{proof}

\vspace{.5cm}

{\bf{Acknowledgment.}} The work of PC was partially supported by NSF grant DMS-
2106528.

\end{document}